\documentclass[11pt, a4paper]{amsart}

\usepackage{amssymb,amsmath,amsthm}
\usepackage{hyperref}
\usepackage{graphicx}
\usepackage{xcolor}
\usepackage{graphicx} 
\hypersetup{backref, pdfpagemode=FullScreen, colorlinks=true,
  citecolor=magenta, linkcolor=cyan, urlcolor=blue}

\usepackage{mathtools}
\mathtoolsset{showonlyrefs}

\newcommand{\Kon}{\mathcal{K}_o^n}
\newcommand{\Ken}{\mathcal{K}_e^n}
\newcommand{\Kn}{\mathcal{K}^n}

\newcommand{\vol}{\mathrm{vol}}

\newcommand{\conv}{\mathrm{conv}}

\newcommand{\aff}{\mathrm{aff}}
\newcommand{\lin}{\mathrm{lin}\,}
\newcommand{\R}{\mathbb{R}}

\newcommand{\N}{\mathbb{N}}

\newcommand{\Sph}{\mathbb{S}}

\newcommand{\W}{\mathrm{W}}
\newcommand{\cura}{\mathrm{C}}
\newcommand{\sura}{\mathrm{S}}
\newcommand{\dW}{\widetilde{\mathrm{W}}}
\newcommand{\dcura}{\widetilde{\mathrm{C}}}

\newcommand{\diff}{\,\mathrm{d}}

\newcommand{\ip}[2]{\left\langle #1,#2\right\rangle}

\newcommand{\va}{{\boldsymbol a}}
\newcommand{\vb}{{\boldsymbol b}}

\newcommand{\ve}{{\boldsymbol e}}

\newcommand{\vt}{{\boldsymbol t}}

\newcommand{\vu}{{\boldsymbol u}}
\newcommand{\vv}{{\boldsymbol v}}

\newcommand{\vx}{{\boldsymbol x}}
\newcommand{\vy}{{\boldsymbol y}}
\newcommand{\vz}{{\boldsymbol z}}

\newcommand{\vtheta}{{\boldsymbol \theta}}

\newcommand{\vnull}{{\boldsymbol 0}}
\newcommand{\ov}{\overline}
\DeclareMathOperator{\dplus}{{\widetilde +}}
\newtheorem{theorem}{Theorem}[section]
\newtheorem{corollary}[theorem]{Corollary}
\newtheorem{lemma}[theorem]{Lemma}

\newtheorem{proposition}[theorem]{Proposition}

\numberwithin{equation}{section}

\usepackage{pdfcomment}

\begin{document}

\title[Necessary  conditions for the even dual Minkowski problem]{Necessary subspace concentration conditions for the even dual Minkowski problem}
\author{Martin Henk}
\author{Hannes Pollehn}
\address{Technische Universit\"at Berlin,
		Institut f\"ur Mathematik,
		Strasse des 17. Juni 136,
		D-10623 Berlin,
		Germany}
\email{henk@math.tu-berlin.de, pollehn@math.tu-berlin.de}
\dedicatory{Dedicated to the memory of Peter M. Gruber}
\keywords{dual curvature measure, cone-volume measure, surface area measure, integral curvature, $L_p$-Minkowski Problem, logarithmic Minkowski problem, dual Brunn-Minkowski theory.}
\subjclass[2010]{52A40, 52A38}
\begin{abstract}
We prove tight subspace concentration inequalities  for  
the dual curvature measures $\dcura_q(K,\cdot)$ of an $n$-dimensional
origin-symmetric convex body for $q\geq n+1$. This supplements former
results obtained in the range $q\leq n$.
\end{abstract}
\maketitle

\section{Introduction}
Let $\Kn$ denote the set of convex bodies in $\R^n$,
i.e., the family of all non-empty convex and compact subsets
$K\subset\R^n$. 
The set of convex bodies having the origin as an interior
point is denoted by $\Kon$ and the subset of origin-symmetric convex bodies, i.e., those sets
$K\in\Kon$ satisfying  $K=-K$,
is denoted by $\Ken$. For $\vx,\vy\in\R^n$,
let $\langle \vx,\vy\rangle$ denote the standard inner product and
$|\vx|=\sqrt{\langle \vx,\vx\rangle}$ the Euclidean norm. We write
$B_n$ for the $n$-dimensional Euclidean unit ball, i.e., 
$B_n=\{ \vx\in\R^n : |\vx|\leq 1 \}$, and $\Sph^{n-1}=\partial B_n$, where $\partial K$ is the set of boundary points of $K\in\Kn$. The $k$-dimensional Hausdorff-measure will be denoted by
$\mathcal{H}^{k}(\cdot)$ and 
instead of $\mathcal{H}^{n}(\cdot)$ we
will also write $\vol(\cdot)$ for the $n$-dimensional volume.  For a
$k$-dimensional set $S\subset\R^n$ we also write $\vol_k(S)$ instead
of $\mathcal{H}^{k}(S)$.  

At the heart of the Brunn-Minkowski theory is the study of the volume
functional with respect to the Minkowski addition of convex bodies. This leads to the theory
of mixed volumes and, in particular, to  the quermassintegrals $\W_i(K)$ of a
convex body $K\in\Kn$.  The latter  may be defined via the classical Steiner
formula, expressing the volume of the Minkowski sum of $K$ and
$\lambda\,B_n$, i.e., the volume of the parallel body of $K$ at distance
$\lambda$ 
as a polynomial in $\lambda$  (cf., e.g., \cite[Sect.~4.2]{Schneider:1993})
\begin{equation}
\vol(K+\lambda\,B_n)=\sum_{i=0}^n \lambda^i\, \binom{n}{i} \W_i(K).
\label{eq:steiner} 
\end{equation}  
A more direct geometric interpretation is given
by Kubota's integral formula (cf., e.g., \cite[Subsect.~5.3.2]{Schneider:1993}),  showing that they are --- up to some
constants  --- the means of the volumes of projections 
\begin{equation} 
\W_{n-i}(K)=\frac{\vol(B_n)}{\vol_i(B_i)}\int_{G(n,i)}\vol_i(K|L)\diff L,
\quad i=1,\dots,n, 
\label{eq:kubota}
\end{equation}  
where integration
is taken with respect to the rotation-invariant probability measure on
the Grassmannian $G(n,i)$ of all $i$-dimensional linear subspaces  and
$K|L$ denotes the image of the orthogonal projection onto $L$.  

A local version of the Steiner formula above leads to two important
series of geometric measures, the  
area measures $\sura_i(K,\cdot)$ and the curvature measures
$\cura_i(K,\cdot)$, $i=0,\dots,n-1$, of a convex body $K$. Here we will
only briefly describe the area measures since with respect to
characterization problems of geometric measures they 
form the   ``primal''  counterpart to the dual curvature measures we
are interested in.

To this end, we denote for 
$\omega\subseteq \Sph^{n-1}$ by $\nu^{-1}_K(\omega)\subseteq\partial
K$  the set of all boundary points  of $K$ having
an outer unit normal in $\omega$. We use the notation
$\nu^{-1}_K$ in order to indicate that for smooth convex bodies
$K$  it is the inverse of the Gauss map assigning to a boundary
point of $K$ its unique outer unit normal.   
Moreover, for
$\vx\in\R^n$ let $r_K(\vx)\in K$ be the point in $K$
closest to $\vx$.  Then for a Borel set
$\omega\subseteq \Sph^{n-1}$ and $\lambda>0$ we consider the 
local  parallel body 
\begin{equation}
B_K(\lambda,\omega) = \left\{\vx\in\R^n : 0< |\vx-r_K(\vx)|\leq \lambda
\text{ and } r_K(\vx)\in\nu_K^{-1}(\omega)\right\}. 
\label{eq:local_parallel}
\end{equation}  
The local Steiner formula expresses   the volume of $B_K(\lambda,\omega)$ as a polynomial in $\lambda$ 
whose  coefficients are (up to constants
depending on $i,n$)  the area measures (cf., e.g., \cite[Sect.~4.2]{Schneider:1993})
\begin{equation}
\vol( B_K(\lambda,\omega)) = \frac{1}{n} \sum_{i=1}^{n} \lambda^i\,
\binom{n}{i} \sura_{n-i}(K,\omega).
\label{eq:local_steiner} 
\end{equation}
$\sura_{n-1}(K,\cdot)$ is also known as the surface area measure of
$K$. 
The area measures  may also be regarded as the (right hand side) differentials of
the quermassintegrals, since for $L\in\Kn$ 
\begin{equation}
\lim_{\epsilon \downarrow 0}\frac{\W_{n-1-i}(K+\epsilon
	L)-\W_{n-1-i}(K)}{\epsilon}=\int_{\Sph^{n-1}}h_L(\vu)\diff\sura_{i}(K,\vu).
\label{eq:differentials}
\end{equation} 
Here $h_L(\cdot)$ denotes the support function of $L$ (cf.~Section
2).  Also observe that $\sura_i(K,\Sph^{n-1})=n\,\W_{n-i}(K)$, $i=0,\dots, n-1$.  

To characterize the  area measures $\sura_i(K,\cdot)$,
$i\in\{1,\dots,n-1\}$,  among the  finite Borel
measures on the sphere is a cornerstone of
the Brunn-Minkowski theory. Today this  problem is known as  the {\em Minkowski--Christoffel} 
problem, since for $i=n-1$ 
it is the classical Minkowski problem and for $i=1$
it is the Christoffel problem.  We refer to \cite[Chapter
8]{Schneider:1993} for more information and references.

There are two far-reaching extensions of the classical Brunn-\-Minkowski theory,
both arising 
basically  by replacing the classical Minkowski-addition
by another additive operation (cf.~\cite{Gardner:2013,Gardner:2014}). The first one is the $L_p$ addition
introduced by Firey (see, e.g.,  \cite{Firey:1962}) which leads to the rich and emerging 
{\em $L_p$-Brunn-Minkowski theory} for which we refer to
\cite[Sect.~9.1, 9.2]{Schneider:1993}). 

The second one, introduced by Lutwak \cite{Lutwak:1975b,Lutwak:1975a},  is based on the radial addition $\dplus$ where $\vx
\dplus\vy = \vx+\vy$  if $\vx,\vy$ are linearly dependent and
$\vnull$ otherwise.  Considering the volume of radial additions leads
to the {\em dual Brunn-Minkowski theory} (cf.~\cite[Sect.~9.3]{Schneider:1993}) 
with dual mixed volumes, and, in
particular, also with dual quermassintegrals $\dW_i(K)$
arising via a dual Steiner formula  (cf.~\eqref{eq:steiner})
\begin{equation}
\vol(K\dplus\lambda\,B_n)=\sum_{i=0}^n \lambda^i\, \binom{n}{i}
\dW_i(K).
\label{eq:dual_steiner} 
\end{equation} 
In general the radial addition of two convex sets is not a convex set,
but the radial addition of two star bodies is again a  star body. This is one of the features of the dual
Brunn-Minkowski theory which makes it so useful.    
The celebrated solution
of the Busemann-Petty problem is amongst the recent successes 
of the  dual Brunn-Minkowski theory, cf. \cite{Gardner:1994,
	GardnerKoldobskySchlumprecht:1999, Zhang:1999}, and it 
also has connections and applications to integral geometry,
Minkowski geometry, and the local theory of Banach spaces. 

In analogy to Kubota's formula \eqref{eq:kubota}  the dual quermassintegrals $\dW_i(K)$  admit the following integral
geometric representation as the means of the volumes of sections (cf.~\cite[Sect. 9.3]{Schneider:1993}) 
\begin{equation}
\dW_{n-i}(K)=\frac{\vol(B_n)}{\vol_i(B_i)}\int_{G(n,i)}\vol_i(K\cap
L)\diff L,
\quad i=1,\dots,n. 
\end{equation} 
There are many more ``dualities'' between the classical and
dual theory, but  there were no dual geometric
measures corresponding to the   area and  curvature measures.
This missing link was recently established in the 
ground-breaking paper \cite{Huang:2016} by Huang, Lutwak, Yang and
Zhang. Let $\rho_K$ be the radial function (cf.~Section 2)
of a convex body $K\in \Kon$. 
Analogous to
\eqref{eq:local_parallel} we consider for a Borel set $\eta\subseteq \Sph^{n-1}$ and $\lambda>0$ the set 
\begin{equation}
\begin{split}
{\widetilde A}_K&(\lambda,\eta) = \\ 
&\left\{\vx\in\R^n\setminus\{\vnull\}  : (1-\rho_K(\vx))\,|\vx|\leq \lambda
\text{ and } \rho_K(\vx)\vx\in\nu_K^{-1}(\eta)\right\}\cup\{\vnull\}. 
\end{split} 
\end{equation}    
Then there also exists  a local Steiner type formula of these  local
dual parallel sets \cite[Theorem 3.1]{Huang:2016} (cf.~\eqref{eq:local_steiner})
\begin{equation}
\vol({\widetilde A}_K(\lambda,\eta))=\sum_{i=0}^n \binom{n}{i}\lambda^{i}\dcura_{n-i}(K,\eta).
\end{equation} 
$\dcura_{i}(K,\cdot)$ is called the {\em $i$th dual curvature
	measure} and they are the counterparts to the curvature  measures
$\cura_i(K,\cdot)$  within the dual Brunn-Minkowski theory. Observe that
$\dcura_i(K,\Sph^{n-1})=\dW_{n-i}(K)$. 
As the  area measures (cf.~\eqref{eq:differentials}), the dual
curvature measures  may also be
considered as differentials of the dual quermassintegrals, even in a
stronger form (see  \cite[Section 4]{Huang:2016}). 
We want to 
point out  that  there are also
dual area measures corresponding to the area measures
in the classical theory (see \cite{Huang:2016}).   

Huang, Lutwak, Yang and Zhang also gave  an explicit 
integral representation of the dual curvature measures  which allowed
them to define more generally for any $q\in \R$ the  $q$th dual
curvature measure  of a convex body $K\in\Kon$ as \cite[Def. 3.2]{Huang:2016} 
\begin{equation}
\dcura_q(K,\eta) = \frac1n
\int\limits_{\alpha_K^\ast(\eta)} \rho_K(\vu)^q
\diff\mathcal{H}^{n-1}(\vu).  
\label{eq:dual_curvature_measure}
\end{equation}
Here  $\alpha_K^\ast(\eta)$
denotes the set of directions $\vu\in \Sph^{n-1}$, such that the boundary
point $\rho_K(\vu)\vu$ belongs to $\nu_K^{-1}(\eta)$.  
The analog to the Minkowski-Christoffel problem in the dual
Brunn-Minkowski theory is  (cf.~\cite[Sect.~5]{Huang:2016}) 

\smallskip
\begin{center}
	\begin{minipage}{0.95\textwidth}{\em The  dual Minkowski problem.} Given  a finite Borel measure $\mu$ on
		$\Sph^{n-1}$ and $q\in \R$. Find necessary and sufficient conditions
		for the existence of a convex body $K\in\Kon$ such that $\dcura_q(K,\cdot)=\mu$.  
	\end{minipage} 
\end{center}
\smallskip 

\noindent
An amazing feature of these dual curvature measures
is that they 
link two other well-known  fundamental geometric measures of a convex body
(cf.~\cite[Lemma 3.8]{Huang:2016}):
when  $q=0$  the dual curvature measure
$\dcura_0(K,\cdot)$ is -- up to a factor of $n$ -- 
Aleksandrov's integral curvature of the polar body of $K$ and for
$q=n$ the dual curvature measure coincides with the cone-volume
measure of $K$ given by 
\begin{equation}
\dcura_n(K,\eta)	= V_K(\eta) = \frac{1}{n} \int\limits_{\nu_K^{-1}(\eta)} \langle \vu,\nu_K(\vu)\rangle \diff\mathcal{H}^{n-1}(\vu).
\end{equation}



Similarly to the Minkowski problem, solving the dual Minkowski problem is equivalent to solving a Monge-Amp\`{e}re type partial differential equation if the measure $\mu$ has a density function $g : \Sph^{n-1}\to \R$. In particular, the dual Minkowski problem amounts to solving the Monge-Amp\`ere  equation
\begin{equation}
\label{Monge-Ampere}
\frac1n\,h(\vx)|\nabla h(\vx) + h(\vx)\,\vx |^{q-n}\det[h_{ij}(\vx) + \delta_{ij}h(\vx)] = g(\vx),
\end{equation}
where $[h_{ij}(\vx)]$ is the Hessian matrix of the (unknown) support function $h$ with respect to an orthonormal
frame on $\Sph^{n-1}$, and $\delta_{ij}$ is the Kronecker delta.

If $\frac1n\,h(\vx)|\nabla h(\vx) + h(\vx)\vx|^{q-n}$ were omitted in \eqref{Monge-Ampere}, then \eqref{Monge-Ampere} would become the partial differential
equation of the classical Minkowski problem, see, e.g., 
\cite{Caffarelli:1990,Cheng:1976,Nirenberg:1953}. If only the factor
$|\nabla h(\vx) + h(\vx)\,\vx|^{q-n}$ were omitted, then
equation \eqref{Monge-Ampere} would become the partial differential equation associated with the 
cone volume measure, the so-called {\sl logarithmic Minkowski problem} (see,
e.g., \cite{Boroczky:2013, Chou:2006}). Due to the gradient component
in \eqref{Monge-Ampere} the dual Minkowski problem is
significantly more challenging than the classical Minkowski problem as
well as the logarithmic Minkowski problem.


The cone-volume measure for convex bodies  has been studied extensively over the last few
years in many different contexts, see, e.g., \cite{Barthe:2005,
	Boroczky:2015c, Boroczky:2012,Boroczky:2013,Boroczky:2015a,
	Gardner:2014,  Gromov:1987, Haberl:2014, He:2006, Henk:2014, Henk:2005, 
	Huang:2016, Ludwig:2010a, Ludwig:2010b,	Lutwak:2005, Lutwak:2010a,
	Lutwak:2010b, Ma:2015, Naor:2007, Naor:2003, Paouris:2012,
	Stancu:2012, Zhu:2014, Zhu:2015}.  
One very important property of the cone-volume measure -- and which
makes it so essential --  is its
$\mathrm{SL}(n)$-invariance, or simply called affine invariance.  It
is also the subject of the central {\em logarithmic Minkowski problem}
which asks  for sufficient and necessary conditions 
of a measure $\mu$ on $\Sph^{n-1}$  to be  the cone-volume measure
$V_K(\cdot)$ of a
convex body $K\in\Kon$. This is the $p=0$ limit case of the general
{\sl $L_p$-Minkowski problem} within the above mentioned $L_p$
Brunn-Minkowski theory for which we refer to \cite{Hug:2005,
	Lutwak:1993, Zhu:2015a} and the references within.

The discrete, planar, even case of the logarithmic Minkowski
problem, i.e., with respect to origin-symmetric convex polygons, was
completely solved by Stancu~\cite{Stancu:2002,Stancu:2003}, and later
Zhu~\cite{Zhu:2014} as well as B\"or\"oczky, Heged{\H{u}}s and
Zhu~\cite{Boroczky:2015b} settled (in particular) the  case when $K$ is a polytope 
whose outer normals are in general position.

In~\cite{Boroczky:2013}, B\"or\"oczky, Lutwak, Yang and Zhang gave a
complete characterization of the cone-volume measure of
origin-symmetric convex bodies among the even measures on the sphere. 
The key feature of such a measure is expressed via the following
condition:  A non-zero,
finite Borel measure $\mu$ on the unit sphere satisfies the 
\emph{subspace concentration condition} if
\begin{equation}
\frac{\mu(\Sph^{n-1}\cap L)}{\mu(\Sph^{n-1})} \leq \frac{\dim L}{n}
\label{eq:subspace_concentration_inequality}
\end{equation}
for every proper subspace $L$ of $\R^n$, and whenever we have equality in~\eqref{eq:subspace_concentration_inequality} for some $L$, there is a subspace $L'$ complementary to $L$, such that $\mu$ is concentrated on $\Sph^{n-1}\cap (L\cup L')$.

Apart from the uniqueness aspect, 
the symmetric case of the logarithmic Minkowski problem is settled.
\begin{theorem}[\cite{Boroczky:2013}]
	A non-zero, finite, even Borel measure $\mu$ on $\Sph^{n-1}$ is
	the cone-volume measure of $K\in \Ken$  if and only if $\mu$ satisfies the subspace concentration condition.
	\label{thm:logarithmic_minkowski_problem}
\end{theorem}
An extension of the
validity of  inequality \eqref{eq:subspace_concentration_inequality}
to centered bodies, i.e.,   bodies whose center of mass is at the
origin, was given  in the discrete case by Henk and Linke \cite{Henk:2014}, and
in the general setting by B\"or\"oczky and
Henk~\cite{Boroczky:2016}. For a related stability result concerning
\eqref{eq:subspace_concentration_inequality} we refer to \cite{Boroczky:2017a}.
In \cite{Chen:2017}, Chen, Li and Zhu proved that also in the non-symmetric logarithmic Minkowski problem the subspace concentration condition is sufficient.


A generalization (up to the equality case) of the sufficiency part of Theorem
\ref{thm:logarithmic_minkowski_problem} to the $q$th dual curvature
measure  for
$q\in(0,n]$ was  given by Huang, Lutwak, Yang and Zhang.
For clarity, we separate their main result into 
the next two theorems.

\begin{theorem}[\protect{\cite[Theorem 6.6]{Huang:2016}}]
	Let $q\in (0,1]$. A non-zero, finite,  even Borel measure $\mu$ on
	$\Sph^{n-1}$ is the $q$th dual curvature measure of a convex
        body $K\in\Kon$ if and only if $\mu$ is not concentrated on any great subsphere.
	\label{thm:subspace_mass0}
\end{theorem}

\begin{theorem}[\protect{\cite[Theorem 6.6]{Huang:2016}}]
	Let $q\in (1,n]$ and let  $\mu$  be a non-zero, finite,  even Borel measure  on
	$\Sph^{n-1}$ satisfying  the {\em  subspace mass inequality}  
	\begin{equation}
	\frac{\mu(\Sph^{n-1}\cap L)}{\mu(\Sph^{n-1})} <
	1-\frac{q-1}{q}\frac{n-\dim L}{n-1}
	\label{eq:subspace_mass} 
	\end{equation}
	for every proper subspace $L$ of $\R^n$. Then there exists an
	$o$-symmetric convex body  $K\in\Ken$ with $\dcura_q(K,\cdot)=\mu$. 
	\label{thm:subspace_mass}
\end{theorem}

Observe that for $q=n$ the inequality \eqref{eq:subspace_mass} becomes
essentially \eqref{eq:subspace_concentration_inequality}. In case that the parameter $q$ is an integer this result was strengthened by Zhao.
\begin{theorem}[\cite{Zhao:2017b}]
	Let $q\in\{1,\dots,n-1\}$ and let  $\mu$  be a non-zero, finite,  even Borel measure  on
	$\Sph^{n-1}$ satisfying
	\begin{equation}
	\frac{\mu(\Sph^{n-1}\cap L)}{\mu(\Sph^{n-1})} < \min\left\{ \frac{\dim L}{q},1\right\},
	\end{equation}
	for every proper subspace $L$ of $\R^n$. Then there exists a
	$o$-symmetric convex body  $K\in\Ken$ with $\dcura_q(K,\cdot)=\mu$.
\end{theorem}
An extension of this result to all $q\in(0,n)$ was very recently given  by
B\"or\"oczky, Lutwak, Yang, Zhang and Zhao~\cite{Boroczky:2017b}. That this subspace
concentration bound is indeed
necessary was shown  by B\"or\"oczky and the authors.
\begin{theorem}[\cite{BoroczkyHenkPollehn:2017}]
Let $K\in\Ken$, $q\in (0,n)$ and let $L\subset \R^n$ be a proper
subspace. Then we have 
\begin{equation}
	\frac{\dcura_q(K,\Sph^{n-1}\cap L)}{\dcura_q(K,\Sph^{n-1})} <
        \min\left\{ \frac{\dim L}{q},1\right\}.
\end{equation}
\end{theorem}
The case $q<0$ (including uniqueness) was completely settled by
Zhao~\cite{Zhao:2017a}. In particular, he proved that  there is no (non-trivial) subspace
concentration 
\begin{theorem}[\cite{Zhao:2017a}] 
Let $q<0$. A non-zero, finite,  even Borel measure $\mu$ on
	$\Sph^{n-1}$ is the $q$th dual curvature measure of a convex
        body $K\in\Kon$ if and only if $\mu$ is not concentrated on any closed hemisphere.
Moreover, $K\in\Kon$ is
uniquely determined.
\end{theorem} 

Our main result treats the range $q\geq n+1$ and here again we have
non-trivial subspace concentration bounds on dual curvature measures
of origin-symmetric convex bodies. 
\begin{theorem}
	Let $K\in\Ken$, $q\geq n+1$, and let $L\subset \R^n$ be a proper subspace of $\R^n$. Then we have 
	\begin{equation}
		\frac{\dcura_q(K,\Sph^{n-1}\cap L)}{\dcura_q(K,\Sph^{n-1})} < \frac{q-n+\dim L}{q}.
		\label{eq:subspace_bound}
	\end{equation}
	\label{thm:subspace_bound}
\end{theorem}
This bound is also optimal.
\begin{proposition}
	Let $q>n$ and $k\in\{1,\ldots,n-1\}$. There exists a sequence of convex bodies $K_l\in\Ken$, $l\in\N$, and a $k$-dimensional subspace $L\subset\R^n$ such that
	\begin{equation}
		\lim_{l\to\infty} \frac{\dcura_q(K_l,\Sph^{n-1}\cap L)}{\dcura_q(K_l,\Sph^{n-1})}=\frac{q-n+k}{q}.
	\end{equation}
	\label{prop:bound_opt}
\end{proposition}

Unfortunately, our approach can not cover the missing range
$q\in(n,n+1)$. One reason is that the following Brunn-Minkowski-type
inequality \eqref{eq:dual_curvature_measure_combination} for moments of the Euclidean norm  which might be of some interest in its
own does not hold  in general
for $0<p<1$. The proof of  Theorem \ref{thm:subspace_bound} heavily relies on this inequality.  

\begin{theorem}
	\label{thm:dual_curvature_measure_combination}
	Let $K_0,K_1\in\Kn$, with $\dim K_0=\dim K_1=k\geq 1$, 
        $\vol_k(K_0)=\vol_k(K_1)$ and their
        affine hulls are parallel. For
        $\lambda\in[0,1]$ let  $K_\lambda=(1-\lambda)K_0+\lambda K_1$.
         Then for $p\geq 1$ 
	\begin{equation}
		\label{eq:dual_curvature_measure_combination}
		\begin{split}
			\int_{K_\lambda} |\vx|^p &\diff\mathcal{H}^k(\vx)+\int_{K_{1-\lambda}} |\vx|^p \diff\mathcal{H}^k(\vx)\\
			&\geq |2\lambda-1|^p \left( \int_{K_0} |\vx|^p \diff\mathcal{H}^k(\vx) +\int_{K_1} |\vx|^p \diff\mathcal{H}^k(\vx) \right)
		\end{split}
	\end{equation}
	with equality if and only if $\lambda\in\{0,1\}$ or $p=1$  and
there exists a $\vu\in\Sph^{n-1}$ such that $K_0, K_1\subset \lin\vu$
and the hyperplane $\{\vx\in\R^n : \ip{\vu}{\vx}=0\}$ separates $K_\lambda$ and $K_{1-\lambda}$.
\end{theorem} 
Here $\lin(\cdot)$ denotes the linear hull operator.
Observe for $p=0$ the inequality also holds true by Brunn-Minkowski
inequality (cf.~\eqref{eq:brunn_minkowski}). We will use the theorem in the special setting $K_0=-K_1$
which then gives 
\begin{corollary}  
\label{cor: dual_curvature_measure_combination}
	Let $K\in\Kn$ with $\dim K=k\geq 1$ and
        let $p\geq 1$. Then for $\lambda\in[0,1]$ 
	\begin{equation}
			\int_{(1-\lambda)K +\lambda(-K)} |\vx|^p \diff\mathcal{H}^k(\vx)\geq
                        |2\lambda-1|^p  \int_{K} |\vx|^p
                        \diff\mathcal{H}^k(\vx), 
	\end{equation}
with equality if and only if  $\lambda\in\{0,1\}$ or $p=1$  and
there exists a $\vu\in\Sph^{n-1}$ such that  $K\subset  \lin\vu$
and the origin is not in the relative interior of the segment
$(1-\lambda)K +\lambda(-K)$.
\end{corollary}

In the planar case we can fill the remaining gap in the range of $q$, i.e., there we will
prove a sharp concentration bound for all $q> 2$.  
\begin{theorem}
	Let $K\in\mathcal{K}_e^2$, $q> 2$, and let $L\subset \R^2$ be a line through the origin. Then we have 
	\begin{equation}
		\frac{\dcura_q(K,\Sph^1\cap L)}{\dcura_q(K,\Sph^1)} < \frac{q-1}{q}.
		\label{eq:subspace_concentration_plane}
	\end{equation}
	\label{thm:subspace_bound_plane}
\end{theorem}

We remark that the logarithmic Minkwoski problem as well as the dual
Minkowski problem are far easier to handle for the special case 
where the measure $\mu$  has a positive continuous density, (where
subspace concentration is trivially satisfied). The singular general
case for measures is substantially more delicate, which involves
measure concentration and requires far more powerful techniques to
solve.


The paper is organized as follows. First we will briefly recall
some basic facts about convex bodies needed in our investigations in
Section 2.  In
Section 3 we will prove Theorem \ref{thm:dual_curvature_measure_combination}, which is one of the main ingredients for the
proof of Theorem~\ref{thm:subspace_bound} given in Section 4 alongside the proof of Proposition~\ref{prop:bound_opt}. Finally,
we discuss the remaining case $q\in(n,n+1)$ in Section 5 and prove Theorem~\ref{thm:subspace_bound_plane}.

\section{Preliminaries}
We recommend the books by Gardner \cite{Gardner:2006}, Gruber \cite{Gruber:2007} and
Schneider \cite{Schneider:1993} as excellent references on  convex geometry.

For a given convex body $K\in\Kn$ the support function $h_K\colon\R^n\to\R$ is defined by
\begin{equation}
h_K(\vx)=\max_{\vy\in K} \langle \vx,\vy\rangle.
\end{equation}
A boundary point $\vx\in\partial K$ is said to have a (not necessarily unique) unit outer normal vector $\vu\in \Sph^{n-1}$ if $\langle \vx,\vu\rangle = h_K(\vu)$. The corresponding supporting hyperplane $\{\vx\in \R^n\colon \langle \vx,\vu\rangle=h_K(\vu)\}$ will be denoted by $H_K(\vu)$. For $K\in\Kon$ the radial function $\rho_K\colon \R^n\setminus\{\vnull\}\to\R$ is given by
\begin{equation}
\rho_K(\vx)=\max\{\rho>0\colon \rho\, \vx\in K\}.
\end{equation}
Note, that the support function and the radial function are
homogeneous of degrees $1$ and $-1$, respectively, i.e., 
\begin{equation} 
h_K(\lambda\,\vx) =\lambda\, h_K(\vx)\text{ and } \rho_K(\lambda\,\vx)=\lambda^{-1}\, \rho_K(\vx),
\end{equation} 
for $\lambda>0$.
We define the reverse radial Gauss image of $\eta\subseteq \Sph^{n-1}$
with respect to a convex body $K\in\Kon$ by
\begin{equation}
\alpha_K^\ast(\eta) = \{\vu\in \Sph^{n-1}\colon \rho_K(\vu)\vu\in H_K(\vv) \text{ for a }\vv\in\eta \}.
\end{equation}
If $\eta$ is a Borel set, then $\alpha_K^\ast(\eta)$ is
$\mathcal{H}^{n-1}$-measurable
(see~\cite[Lemma~2.2.11.]{Schneider:1993}) and so the  $q$th dual
curvature measure given in~\eqref{eq:dual_curvature_measure} is well defined. We will need the following identity.
\begin{lemma}[\protect{\cite[Lemma 2.1]{BoroczkyHenkPollehn:2017}}]
	Let $K\in\Kon$, $q>0$ and $\eta\subseteq \Sph^{n-1}$ a Borel set.
        Then 
	\begin{equation}
	\dcura_q(K,\eta) = \frac{q}{n} \int\limits_{\{\vx\in K : \vx/|\vx|\in\alpha_K^\ast(\eta)\}} |\vx|^{q-n} \diff\mathcal{H}^n(\vx).
	\end{equation}
	\label{lem:curvature_measure_euclidean_coordinates}
\end{lemma}

Let $L$ be a linear subspace of $\R^n$. We write $K|L$ to denote the image of the
orthogonal projection of $K$ onto $L$  and $L^\perp$ for the subspace
orthogonal to $L$. The linear (affine, convex) hull of a set $S\subseteq \R^n$ is
denoted by $\lin S$ ($\aff S$, $\conv S$), and for $\vv\in\R^n$ we write instead of $(\lin
\vv)^\perp$ just $\vv^\perp$.

As usual, for two subsets  $A, B\subseteq \R^n$ and reals
$\alpha,\beta \geq 0$ the  Minkowski combination is defined by 
\begin{equation}
\alpha A+\beta B = \{\alpha \va+\beta \vb : \va\in A,\vb\in
B\}. 
\end{equation} 

By the well-known Brunn-Minkowski inequality we know that the $n$th root 
of the volume of the Minkowski combination  is a concave function. 
More precisely, for two convex bodies
$K_0,K_1\subset\R^n$ and for $\lambda\in[0,1]$  we have 
\begin{equation}
\vol_n((1-\lambda)K_0+\lambda K_1)^{1/n}\geq (1-\lambda)\vol_n(K_0)^{1/n}+ \lambda \vol_n(K_1)^{1/n},
\label{eq:brunn_minkowski}
\end{equation}
where $\vol_n(\cdot)=\mathcal{H}^n(\cdot)$ denotes the $n$-dimensional
Hausdorff measure.  We have equality in \eqref{eq:brunn_minkowski}
for some $0<\lambda<1$ if and only if $K_0$ and $K_1$ lie in parallel
hyperplanes  or they are homothetic, i.e., there exist a $\vt\in\R^n$ and $\mu\geq 0$ such that $K_1=\vt+\mu\,K_0$ (see, e.g., \cite{Gardner:2002}, \cite[Sect. 6.1]{Schneider:1993}).

\section{A Brunn-Minkowski type inequality for moments of the
  Euclidean norm }
In this section we will prove Theorem 
\ref{thm:dual_curvature_measure_combination} and  
we start by recalling a variant of the well-known Karamata 
inequality which often appears in the context of Schur-convex functions.

\begin{theorem}[Karamata's inequality, see, e.g., \protect{\cite[Theorem 1]{Kadelburg:2005}}]
	Let $D\subseteq\R$ be convex  and let $f\colon D\to\R$ be a
        non-decreasing, convex function. Let $x_1,\ldots,x_k,y_1,\ldots,y_k\in D$ such that
	\begin{enumerate}
		\item $x_1\geq x_2\geq\ldots\geq x_k,$
		\item $y_1\geq y_2\geq\ldots\geq y_k,$
		\item $x_1+x_2+\ldots+x_i \geq y_1+y_2+\ldots+y_i$ for all $i=1,\ldots,k,$
	\end{enumerate}
	then
	\begin{equation}
	\label{eq:karamata}
	f(x_1)+f(x_2)+\ldots+f(x_k)\geq f(y_1)+f(y_2)+\ldots+f(y_k).
	\end{equation}
	If $f$ is strictly convex, then equality in
        \eqref{eq:karamata} holds if and only if $x_i=y_i$,
        $i=1,\ldots,k$.
\label{thm:karamata}          
\end{theorem}

As a consequence we obtain an estimate for the value of powers of convex combinations of real numbers.

\begin{lemma}
	\label{lem:inequality_karamata}
	Let $p\geq 1$, $z,\bar{z}\in\R$ and $\lambda\in [0,1]$. Then
	\begin{equation}
		|\lambda z+(1-\lambda)\bar{z}|^p+|\lambda \bar{z}+(1-\lambda)z|^p\geq |2\lambda-1|^p (|z|^p+|\bar{z}|^p)
	\end{equation}
	and equality holds if and only if at least one of the following statements is true:
	\begin{enumerate}
		\item $\lambda\in\{0,1\}$,
		\item $\bar{z}=-z$,
		\item $p=1$, $z\bar{z}<0$ and $\max\{\lambda,1-\lambda\}\geq \frac{\max\{|z|,|\bar{z}|\}}{|z|+|\bar{z}|}$.
		\label{itm:karamata_equality_p_1}
	\end{enumerate}
\end{lemma}
\begin{proof}
	By symmetry we may assume $\lambda\geq\frac12$, $|z|\geq |\bar{z}|$ and $z\geq 0$. Write
	\begin{align*}
	x_1 &= |\lambda z+(1-\lambda)\bar{z}|,\\
	x_2 &= |\lambda \bar{z}+(1-\lambda)z|,\\
	y_1 &= (2\lambda-1)z,\\
	y_2 &= (2\lambda-1)|\bar{z}|.
	\end{align*}
	We want to apply Karamata's inequality with $D=\R_{\geq 0}$
        and $f(t)=t^p$. We readily have
	\begin{gather*}
	y_1 \geq y_2,\\
	x_1^2-x_2^2 = (2\lambda-1)(z^2-\bar{z}^2)\geq 0,
	\end{gather*}
	and since  $\bar{z}\geq -z$ we also have 
\begin{equation*}
                x_1\geq y_1.
\end{equation*} 
	It remains to show that $x_1+x_2\geq y_1+y_2$. The triangle inequality gives
	\begin{equation}
	|(\lambda z+(1-\lambda)\bar{z})\pm (\lambda \bar{z}+(1-\lambda)z)| \leq x_1+x_2.
	\end{equation}
	Hence
	\begin{align*}
	x_1+x_2 &\geq \max\{ (2\lambda-1) |z-\bar{z}|,|z+\bar{z}| \}\\
	&\geq (2\lambda-1) \max\{ |z-\bar{z}|,|z+\bar{z}| \}\\
	\notag
	&= (2\lambda-1) (z+|\bar{z}|)\\
	\notag
	&= y_1+y_2
	\end{align*}
	and Karamata's inequality \eqref{eq:karamata} 
        yields $x_1^p+x_2^p\geq y_1^p+y_2^p$, i.e., the inequality of
        the lemma. 
	
	Suppose now we have equality and as before we assume
        $\lambda\geq\frac12$, $|z|\geq |\bar{z}|$ and $z\geq 0$.  
 Let  $\lambda <1$ and first let  $p>1$. In this case the equality
        condition of Karamata's inequality asserts $x_1=y_1$ and since
        $\bar{z}\geq -z$, $\lambda\geq \frac12$ we conclude $\bar{z}=-z$.


	
	Now suppose $p=1$, $\lambda<1$  and $z\neq -\bar{z}$. In view
        of our assumptions we have $\lambda z+(1-\lambda)\bar{z}\geq 0$ and so if 
    \begin{equation*}
     (2\lambda-1)(z+|\bar z|)=y_1+y_2=x_1+x_2=  \lambda
     z+(1-\lambda)\bar{z} + |\lambda \bar{z}+(1-\lambda)z| 
\label{eq:last_equality} 
    \end{equation*}
then $\lambda \bar{z}+(1-\lambda)z \leq 0$ and $\bar z< 0$. Thus   
	\begin{equation*}
		\lambda \geq \frac{z}{z-\bar{z}}.
	\end{equation*}
        On the other hand,  condition
        \eqref{itm:karamata_equality_p_1} implies $x_1+x_2=y_1+y_2$.

\end{proof}

The next  lemma allows us to replace spheres 
appearing as level sets of the norm function by hyperplanes. It
appeared first in Alesker \cite{Alesker:1995} and for more explicit versions of it we refer to 
\cite[Lemma 2.1]{Paouris:2003} and
\cite[(10.4.2)]{ArtsteinGiannopoulosMilman:2015}.

\begin{lemma}
	\label{lem:interchange_norm_and_scalar_product}
	Let $p\geq 1$. There is a constant $c=c(n,p)$ such that for every $\vx\in\R^n$
	\begin{equation}
		|\vx|^p = c\cdot \int_{\Sph^{n-1}} |\langle \vx,\vtheta\rangle |^p \diff\mathcal{H}^{n-1}(\vtheta).
	\end{equation}
\end{lemma}

Now we can give the proof of Theorem
\ref{thm:dual_curvature_measure_combination}. After applying Lemma
\ref{lem:interchange_norm_and_scalar_product} we basically follow the
Kneser-S\"uss proof of the Brunn-Minkowski inequality as given in 
\cite[Proof of Theorem 7.1.1]{Schneider:1993}, and
then at the end we use Lemma \ref{lem:inequality_karamata}.

\begin{proof}[Proof of Theorem \ref{thm:dual_curvature_measure_combination}]
	Without loss of generality  we assume that
        $\vol_k(K_0)=\vol(K_1)=1$. 
        The Brunn-Minkowski-inequality
        \eqref{eq:brunn_minkowski} gives in this setting for any
        $\lambda\in [0,1]$ 
        \begin{equation}
          \vol_k(K_\lambda) \geq 1.   
\label{eq:sp_bm} 
        \end{equation}

In order to prove the desired inequality 
        \eqref{eq:dual_curvature_measure_combination} we first
        substitute there the integrand  $|\vx|^p$  via
        Lemma~\ref{lem:interchange_norm_and_scalar_product} which
        leads after an application of Fubini to the equivalent
        inequality 
	\begin{equation}
	\begin{split}
		\int_{\Sph^{n-1}} &\left( \int_{K_\lambda} |\langle \vx,\vtheta\rangle|^p \diff\mathcal{H}^k(\vx) + \int_{K_{1-\lambda}} |\langle \vx,\vtheta\rangle|^p \diff\mathcal{H}^k(\vx) \right) \diff\mathcal{H}^{n-1}(\vtheta)\\
		&\geq |2\lambda-1|^p \times \\ &  \int_{\Sph^{n-1}}
                \left( \int_{K_0} |\langle \vx,\vtheta\rangle|^p
                  \diff\mathcal{H}^k(\vx) + \int_{K_1} |\langle
                  \vx,\vtheta\rangle|^p \diff\mathcal{H}^k(\vx)
                \right) \diff\mathcal{H}^{n-1}(\vtheta).
\label{eq:to-show1}
	\end{split}
	\end{equation}
Obviously in order to prove \eqref{eq:to-show1} it suffices to verify
for every $\vtheta\in \Sph^{n-1}$ 
	\begin{equation}
	\begin{split}
		\int_{K_\lambda} &|\langle \vx,\vtheta\rangle|^p \diff\mathcal{H}^k(\vx) + \int_{K_{1-\lambda}} |\langle \vx,\vtheta\rangle|^p \diff\mathcal{H}^k(\vx)\\
		&\geq |2\lambda-1|^p \left( \int_{K_0} |\langle \vx,\vtheta\rangle|^p \diff\mathcal{H}^k(\vx) + \int_{K_1} |\langle \vx,\vtheta\rangle|^p \diff\mathcal{H}^k(\vx) \right).
	\end{split}
\label{eq:to-show2}
	\end{equation}

Let $\vtheta\in \Sph^{n-1}$ and for $\alpha\in\R$ denote
\begin{gather*}
	H(\vtheta,\alpha) = \{ \vx\in\R^n : \langle\vx,\vtheta\rangle = \alpha \},\\
	H^-(\vtheta,\alpha) = \{ \vx\in\R^n : \langle\vx,\vtheta\rangle \leq \alpha \}.
\end{gather*}
First suppose that $K_0$ lies in a hyperplane parallel to $\vtheta^\perp$
(and therefore $K_\lambda$ for $\lambda\in [0,1]$), i.e., $K_i\subset
H(\vtheta,\alpha_i)$, $\alpha_i\in\R$,
$i\in\{0,1\}$. By \eqref{eq:sp_bm} and
Lemma~\ref{lem:inequality_karamata} we find 
	\begin{align*}
		\int_{K_\lambda} &|\langle \vx,\vtheta\rangle|^p \diff\mathcal{H}^k(\vx) + \int_{K_{1-\lambda}} |\langle \vx,\vtheta\rangle|^p \diff\mathcal{H}^k(\vx)\\
		&=\vol_k(K_\lambda)|(1-\lambda)\alpha_0+\lambda\alpha_1|^p + \vol_k(K_{1-\lambda})|\lambda\alpha_0+(1-\lambda)\alpha_1|^p\\
		&\geq 
                  |(1-\lambda)\alpha_0+\lambda\alpha_1|^p + |\lambda\alpha_0+(1-\lambda)\alpha_1|^p
          \\
		&\geq |2\lambda-1|^p 
                  (|\alpha_0|^p+|\alpha_1|^p)\\
&= |2\lambda-1|^p 
                  (\vol_k(K_0)|\alpha_0|^p+\vol_k(K_1)|\alpha_1|^p)\\
		&= |2\lambda-1|^p \left( \int_{K_0} |\langle \vx,\vtheta\rangle|^p \diff\mathcal{H}^k(\vx) + \int_{K_1} |\langle \vx,\vtheta\rangle|^p \diff\mathcal{H}^k(\vx) \right).
	\end{align*}

Thus, in the following  we may assume that $K_0\not\subseteq \vv+\vtheta^\perp$ for all
$\vv\in\R^n$ and hence,  $-h_{K_i}(-\vtheta)< h_{K_i}(\vtheta)$ for
$i=0,1$. For $t\in\R$ and for $i=0,1$ we set 
	\begin{align*}
		v_i(t) &= \vol_{k-1}(K_i\cap H(\vtheta,t)),\\
		w_i(t) &= \vol_k(K_i\cap H^-(\vtheta,t)),
	\end{align*}
	so that $w_i(t)=\int_{-\infty}^t v_i(\zeta)\diff\zeta$,
        $i\in\{0,1\}$. On $(-h_{K_i}(-\vtheta),h_{K_i}(\vtheta))$ the
        function $v_i$ is continuous and hence $w_i$ is
        differentiable.  For $i\in\{0,1\}$ let $z_i$ be the inverse
        function of $w_i$. Then $z_i$ is differentiable with 
\begin{equation}
\label{eq:derivative}  
z_i'(\tau)=\frac{1}{w_i'(z_i(\tau))}=\frac{1}{v_i(z_i(\tau))}
\end{equation} 
 for $\tau\in (0,1)$. Writing $z_\mu(\tau)=(1-\mu)z_0(\tau)+\mu z_1(\tau)$ for $\mu\in [0,1]$ we have
	\begin{align}
		\int_{K_\mu} |\langle \vx,\vtheta\rangle|^p \diff\mathcal{H}^k(\vx) &= \int_{\R} |t|^p \vol_{k-1}(K_\mu \cap H(\vtheta,t)) \diff t\\
		&= \int_0^1 |z_\mu(\tau)|^p \vol_{k-1}(K_\mu \cap
                  H(\vtheta,z_\mu(\tau))) z_\mu'(\tau) \diff \tau.
\label{eq:step_bm} 
	\end{align}
	Since $K_\mu \cap H(\vtheta,z_\mu(\tau)) \supseteq (1-\mu)
        (K_0\cap H(\vtheta,z_0(\tau))) + \mu(K_1\cap
        H(\vtheta,z_1(\tau)))$ we may apply the Brunn-Minkowski
        inequality to the latter set  and together with \eqref{eq:derivative} we get 
	\begin{equation}
	\label{eq:BM_sections}
	\begin{split}
		\vol_{k-1}(&K_\mu \cap H(\vtheta,z_\mu(\tau)))
                z_\mu'(\tau)\\[1ex]
&\geq \vol_{k-1}\left( (1-\mu)\,
        (K_0\cap H(\vtheta,z_0(\tau))) + \mu\,(K_1\cap
        H(\vtheta,z_1(\tau)))\right)\,  z_\mu'(\tau)\\ 
		&\geq \left[ (1-\mu) v_0(z_0(\tau))^{\frac{1}{k-1}}+\mu v_1(z_1(\tau))^{\frac{1}{k-1}} \right]^{k-1} \left[ \frac{1-\mu}{v_0(z_0(\tau))}+\frac{\mu}{v_1(z_1(\tau))} \right]\\
		&\geq \left[ v_0(z_0(\tau))^{\frac{1-\mu}{k-1}} v_1(z_1(\tau))^{\frac{\mu}{k-1}} \right]^{k-1} \left[ v_0(z_0(\tau))^{-(1-\mu)} v_1(z_1(\tau))^{-\mu} \right]\\
		&=1,
	\end{split}
	\end{equation}
	where for the last inequality we used  the weighted
        arithmetic/geometric-mean inequality.	Therefore,
        \eqref{eq:step_bm} and \eqref{eq:BM_sections} yield to 
	\begin{equation}
	\begin{split}
		\int_{K_\lambda} &|\langle \vx,\vtheta\rangle|^p \diff\mathcal{H}^k(\vx) + \int_{K_{1-\lambda}} |\langle \vx,\vtheta\rangle|^p \diff\mathcal{H}^k(\vx)\\
		&\geq \int_0^1 |z_\lambda(\tau)|^p + |z_{1-\lambda}(\tau)|^p \diff\tau\\
		&= \int_0^1 |\lambda z_1(\tau)+(1-\lambda) z_0(\tau)|^p + |\lambda z_0(\tau)+(1-\lambda) z_1(\tau)|^p \diff\tau.
	\end{split}
	\end{equation}
	Next in order to estimate the integrand  we use
        Lemma~\ref{lem:inequality_karamata} and then we substitute back
        via \eqref{eq:derivative} 
	\begin{align*}
		\int_{K_\lambda} &|\langle \vx,\vtheta\rangle|^p
                                   \diff\mathcal{H}^k(\vx) +
                                   \int_{K_{1-\lambda}} |\langle
                                   \vx,\vtheta\rangle|^p
                                   \diff\mathcal{H}^k(\vx)\\
&\geq \int_0^1 |\lambda z_1(\tau)+(1-\lambda) z_0(\tau)|^p + |\lambda z_0(\tau)+(1-\lambda) z_1(\tau)|^p \diff\tau\\
		\stepcounter{equation}\tag{\theequation}\label{eq:karamata_applied}
		&\geq \int_0^1 |2\lambda-1|^p \Big(|z_0(\tau)|^p+|z_1(\tau)|^p\Big) \diff\tau\\
		&= |2\lambda-1|^p \left( \int_0^1 |z_0(\tau)|^p \diff\tau + \int_0^1 |z_1(\tau)|^p \diff\tau \right)\\
		&= |2\lambda-1|^p \left( \int_0^1 |z_0(\tau)|^p v_0(z_0(\tau))\cdot z_0'(\tau) \diff\tau \right.\\
		&\qquad\qquad\qquad \left. + \int_0^1 |z_1(\tau)|^p v_1(z_1(\tau))\cdot z_1'(\tau) \diff\tau \right)\\
		&= |2\lambda-1|^p \left( \int_{\R} |t|^p \vol_{k-1}(K_0\cap H(\vtheta,t)) \diff t \right.\\
		&\qquad\qquad\qquad \left. + \int_{\R} |t|^p \vol_{k-1}(K_1\cap H(\vtheta,t)) \diff t \right)\\
		&= |2\lambda-1|^p \left( \int_{K_0} |\langle \vx,\vtheta\rangle|^p \diff\mathcal{H}^k(\vx) + \int_{K_1} |\langle \vx,\vtheta\rangle|^p \diff\mathcal{H}^k(\vx) \right).
	\end{align*}
Hence we have shown \eqref{eq:to-show2} and thus \eqref{eq:dual_curvature_measure_combination}.

Now suppose that equality holds in
\eqref{eq:dual_curvature_measure_combination} for two $k$-dimensional
convex bodies $K_0,K_1$ of $k$-dimensional volume 1. We may assume  $\lambda\in
(0,1)$. Then we also have equality in \eqref{eq:to-show2} for any
$\vtheta\in \Sph^{n-1}$ and  
we choose a $\vtheta\in \Sph^{n-1}$ such that $K_0\not\subseteq
\vv+\vtheta^\perp$ for any $\vv\in\R^n$. 

Then the equality in
\eqref{eq:to-show2} implies equality in \eqref{eq:BM_sections}
and thus 
	\begin{equation}
\begin{split} 
		\vol_k(K_\lambda) & = \int_0^1 \vol_{k-1}(K_\lambda \cap
                H(\vtheta,z_\lambda(\tau))) z_\lambda'(\tau)
                \diff\tau\\ 
                &=1 =\lambda\vol_k(K_0)+(1-\lambda)\vol_k(K_1).
\end{split} 
	\end{equation}
Hence, by the equality conditions of the Brunn-Minkowski inequality,
$K_0$ and $K_1$ are homothets and we conclude $K_0=\vv+K_1$ for some
$\vv\in\R^n$. Thus  for $\tau\in [0,1]$ 
	\begin{equation}
\label{eq:z_compare} 
		z_0(\tau) = z_1(\tau)+\langle \vtheta,\vv\rangle.
	\end{equation}
	
        Since we must also  have equality in \eqref{eq:karamata_applied} 
         the equality conditions (ii), (iii) of
        Lemma~\ref{lem:inequality_karamata} can be applied to
        $z_i(\tau)$, $i=0,1$.  
        If $p>1$ then  Lemma~\ref{lem:inequality_karamata} (ii)
        implies $z_0(\tau)=-z_1(\tau)$ for $\tau\in[0,1]$.  
        Together with \eqref{eq:z_compare}, however, 
        we get the contradiction that $z_0(\tau)=\frac12 \langle\vtheta,\vv\rangle$ is constant. 
        
        Thus we must have $p=1$ and in this case we get from
        Lemma~\ref{lem:inequality_karamata} (iii) and \eqref{eq:z_compare}   
\begin{align} 
\label{eq:equality}
      0&\geq z_0(\tau)z_1(\tau)=z_0(\tau)(z_0(\tau)-\langle
                \vtheta,\vv\rangle), \\ 
        \max\{\lambda,1-\lambda\}&\geq
  \frac{\max\{|z_0(\tau)|,|z_1(\tau)|\}}{|z_0(\tau)|+|z_1(\tau)|}. 
\label{eq:equality_two}
\end{align} 
        For a sufficiently small positive $\varepsilon$ the right   
	hand side of \eqref{eq:equality}, as a function in $z_0$, is positive on
        some interval $(-\varepsilon,0)$ or $(0,\varepsilon)$.  If 
\begin{equation}
        h_{K_0}(-\vtheta),\,h_{K_0}(\vtheta) > 0 
\label{eq:choosetheta} 
\end{equation} 
then the domain $[-h_{K_0}(-\vtheta),h_{K_0}(\vtheta)]$  of
$z_0(\tau)$ contains a sufficiently small neighborhood of $0$ which
then would  contradict \eqref{eq:equality}.

Hence we can assume that there exists no $\vtheta\in \Sph^{n-1}$
satisfying \eqref{eq:choosetheta} which means that there exists no two
points in $K_0$ which can be strictly separated by a hyperplane containing
$\vnull$. 
Thus we can assume that $K_0$ is a segment contained in a line $\lin \vu$
for some $\vu\in \Sph^{n-1}$ and the origin is not a relative interior
point of $K_0$. Interchanging the roles of $K_0$ and
$K_1$ in the argumentation above leads to the same conclusion for $K_1$ and
since the affine hulls of $K_0$ and $K_1$ are parallel we also have $K_1\subset\lin\vu$.

So without loss of generality let  $K_0=[\alpha_0,\alpha_0+1]\cdot\vu$, $\alpha_0\geq 0$, 
and $K_1=[\alpha_1,\alpha_1+1]\cdot \xi\vu$ with $\alpha_1\geq 0$ and
$\xi\in\{\pm 1\}$.  We may choose $\vtheta=\vu$ and 
the inequality $0\geq z_0(\tau)z_1(\tau)$ (cf.~\eqref{eq:equality})
gives $\xi=-1$,  i.e., $K_1=[-\alpha_1-1,-\alpha_1]\cdot\vu$. It remains to show that equality in \eqref{eq:dual_curvature_measure_combination} in this setting is equivalent to $\vu^\perp$ separating $K_\lambda$ and $K_{1-\lambda}$.
        
 By symmetry we may assume $\lambda\geq\frac12$ and by the choice of
 $\vtheta$ we get  
        \begin{align*}
        	z_0(\tau)&= \alpha_0+\tau,\\
        	z_1(\tau)&= -(\alpha_1+1-\tau).
        \end{align*}
       The equality condition \eqref{eq:equality_two} yields 
        \begin{equation}
        	\lambda \geq \frac{\max\{\alpha_0+\tau,\alpha_1+1-\tau\}}{\alpha_0+\alpha_1+1}
        \end{equation}
        for all $\tau\in [0,1]$ and  so
        \begin{equation}
      	  \lambda \geq \frac{\max\{\alpha_0,\alpha_1\}+1}{\alpha_0+\alpha_1+1}.
        \end{equation}
        Therefore $(1-\lambda)(\alpha_0+1)-\lambda\alpha_1\leq 0$ and
        $\lambda\alpha_0-(1-\lambda)(\alpha_1+1)\geq 0$, which in turn
        gives 
\begin{equation} 
        \begin{split}
        	K_\lambda &= [(1-\lambda)\alpha_0-\lambda(\alpha_1+1),\,(1-\lambda)(\alpha_0+1)-\lambda\alpha_1] \cdot\vu \subset H^-(\vu,0),\\[1ex]
        	K_{1-\lambda} &=
                                [\lambda\alpha_0-(1-\lambda)(\alpha_1+1),\,
                                \lambda(\alpha_0+1)-(1-\lambda)\alpha_1]\cdot\vu
                                \subset H^-(-\vu,0).
        \end{split}
\label{eq:repre}
\end{equation} 
Thus $\vu^\perp$ separates $K_\lambda$ and $K_{1-\lambda}$.  It remains
to show that  this
separating property also implies equality in
\eqref{eq:dual_curvature_measure_combination}.  If there exits such an
$\vu$ then we may assume that $K_\lambda$ and $K_{1-\lambda}$ are given as in \eqref{eq:repre}
and so 
        \begin{align*}
        	\int_{K_\lambda} |\vx| &\diff\mathcal{H}(\vx)+\int_{K_{1-\lambda}} |\vx| \diff\mathcal{H}(\vx)\\
        	&=\int\limits_{(1-\lambda)\alpha_0-\lambda(\alpha_1+1)}^{(1-\lambda)(\alpha_0+1)-\lambda\alpha_1} -t \diff t + \int\limits_{\lambda\alpha_0-(1-\lambda)(\alpha_1+1)}^{\lambda(\alpha_0+1)-(1-\lambda)\alpha_1} t\diff t\\
        	&= \frac12 \Big[ (\lambda(\alpha_0+1)-(1-\lambda)\alpha_1)^2 - (\lambda\alpha_0-(1-\lambda)(\alpha_1+1))^2\\
        	&\qquad - ((1-\lambda)(\alpha_0+1)-\lambda\alpha_1)^2 + ((1-\lambda)\alpha_0-\lambda(\alpha_1+1))^2 \Big]\\
        	&= (2\lambda-1)(\alpha_0+\alpha_1+1)\\
        	&= (2\lambda-1) \cdot \frac12 \Big[ (\alpha_0+1)^2-\alpha_0^2+(\alpha_1+1)^2-\alpha_1^2 \Big]\\
        	&= (2\lambda-1) \left( \int_{K_0} |\vx| \diff\mathcal{H}(\vx)+\int_{K_1} |\vx| \diff\mathcal{H}(\vx) \right).
        \end{align*}
\end{proof}

We remark that the factor $|2\lambda-1|^p$ in Theorem
\ref{thm:dual_curvature_measure_combination} (as well as in Corollary \ref{cor: dual_curvature_measure_combination})
cannot be replaced by a smaller one. For if, let $C\in\Ken$, $\vu\in
\Sph^{n-1}$, $\rho\in\R_{>0}$, $K_0=C+\rho\cdot \vu$ and $K_1=-K_0$.
Then, for  every $\mu\in[0,1]$ we have 
\begin{gather*}
\int_{K_\mu} |\vx|^p \diff\mathcal{H}^n(\vx) = \rho^p\int_{C} |\rho^{-1}\vx+(2\mu-1)\vu|^p \diff\mathcal{H}^n(\vx).
\end{gather*}
Moreover, the symmetry of $C$ gives
$K_{1-\mu}=-K_{\mu}$ and 
therefore, for a given $\lambda\in[0,1]$ we conclude 
\begin{equation}
\lim_{\rho\to\infty} \frac{\int_{K_{1-\lambda}} |\vx|^p
  \diff\mathcal{H}^n(\vx)}{\int_{K_1}
  |\vx|^p\diff\mathcal{H}^n(\vx)}=	\lim_{\rho\to\infty} \frac{\int_{K_\lambda} |\vx|^p \diff\mathcal{H}^n(\vx)}{\int_{K_0} |\vx|^p\diff\mathcal{H}^n(\vx)}=|2\lambda-1|^p.
\end{equation}

\section{Proof of Theorem~\ref{thm:subspace_bound}}
Now we are ready to prove 
Theorem~\ref{thm:subspace_bound}. 
We use Fubini's theorem to decompose the dual curvature measure of
$K\in\Ken$ into
integrals over sections with affine planes orthogonal to the given
subspace $L$. In order to compare these
integrals with the corresponding integrals over sections associated to
the dual curvature measure of the Borel set $\Sph^{n-1}\cap L$ we
apply 
Corollary~\ref{cor: dual_curvature_measure_combination}. 
\begin{proof}[Proof of Theorem~\ref{thm:subspace_bound}] Let $\dim
  L=k\in\{1,\dots,n-1\}$, and for  
	 $\vy\in K|L$ let
        $\ov{\vy}=\rho_{K|L}(\vy)\vy\in\partial(K|L)$, 
	\begin{align*}
		K_\vy &= K\cap (\vy+L^\perp),\\
		M_{\ov{\vy}} &= \conv\{K_\vnull, K_{\ov{\vy}}\}.
	\end{align*}
	By Lemma~\ref{lem:curvature_measure_euclidean_coordinates}, 
	Fubini's theorem and the fact that $M_{\ov{\vy}}\cap(\vy+L^\perp)\subseteq K_\vy$ we may write 
	\begin{equation}
	\begin{split} 
		\tilde{C}_q(K,\Sph^{n-1}) &= \frac{q}{n} \int\limits_{K|L} \left(\,\int\limits_{K_\vy} |\vz|^{q-n} \diff\mathcal{H}^{n-k}(\vz) \right)\diff\mathcal{H}^k(\vy)\\
		&\geq \frac{q}{n} \int\limits_{K|L} \left(\,
		\int\limits_{M_{\ov{\vy}}\cap (\vy+L^\perp)} |\vz|^{q-n}
		\diff\mathcal{H}^{n-k}(\vz)
		\right)\diff\mathcal{H}^k(\vy).
	\end{split}
	\label{eq:first_step} 
	\end{equation}
	In order to estimate the inner integral we fix  a $\vy\in K|L$, $\vy\neq\vnull$, and for abbreviation we set
	$\tau=\rho_{K|L}(\vy)^{-1}\leq 1$. Then, by the symmetry of $K$ we find 
	\begin{align*}
		M_{\ov{\vy}}\cap(\vy+L^\perp) \supseteq& \tau (K_{\ov{\vy}})+(1-\tau)(K_\vnull)\\
		\supseteq& \tau (K_{\ov{\vy}})+(1-\tau)\left( \frac12  K_{\ov{\vy}}  + \frac12 ( - K_{\ov{\vy}} ) \right)\\
		=& \frac{1+\tau}{2}K_{\ov{\vy}}+\frac{1-\tau}{2}(-K_{\ov{\vy}}). 
	\end{align*}
	Hence, $M_{\ov{\vy}}\cap(\vy+L^\perp)$ contains a convex combination
	of a set and its reflection at the origin. This allows us to
        apply Corollary ~\ref{cor: dual_curvature_measure_combination}
        from which we  obtain 
	\begin{align*}
		\int\limits_{M_{{\ov\vy}}\cap(\vy+L^\perp)} |\vz|^{q-n} \diff\mathcal{H}^{n-k}(\vz) &\geq \int\limits_{\frac{1+\tau}{2}K_{\ov{\vy}}+\frac{1-\tau}{2}(-K_{\ov{\vy}})} |\vz|^{q-n} \diff\mathcal{H}^{n-k}(\vz)\\
		&\geq \tau^{q-n}\int\limits_{K_{\ov{\vy}}} |\vz|^{q-n}\diff\mathcal{H}^{n-k}(\vz)
	\end{align*}
	for every $\vy\in K|L$, $\vy\neq\vnull$. By the equality
        characterization in Corollary \ref{cor:
          dual_curvature_measure_combination}  the  last inequality is
        strict  whenever $\tau<1$, i.e., $\vy$ belongs to the relative
        interior of $K|L$. 
	Together with \eqref{eq:first_step} we obtain the lower bound  
	\begin{equation}
		\dcura_q(K,\Sph^{n-1}) > \frac{q}{n} \int\limits_{K|L}\left(\, \rho_{K|L}(\vy)^{n-q}\int\limits_{K_{\ov{\vy}}} |\vz|^{q-n}\diff\mathcal{H}^{n-k}(\vz)\right) \diff\mathcal{H}^k(\vy).
		\label{eq:lower_bound} 
	\end{equation}	
	In order to evaluate $\dcura_q(K,\Sph^{n-1}\cap L)$ we note that for  $\vx\in
	K$  we have  $\vx/|\vx|\in\alpha_K^\ast(\Sph^{n-1}\cap L)$ if and only if
	the boundary point  $\rho_K(\vx)\vx$ has an outer unit normal in
	$L$. Hence,   
	\begin{align}
	\begin{split}
		\{\vnull\} &\cup \left\{\vx\in K  :
		\vx/|\vx|\in\alpha_K^\ast(\Sph^{n-1}\cap L) \right\}
              \\ &= 
		\bigcup_{\vv\in \partial(K|L)}  \conv\{\vnull,K_\vv
                \},  
	\end{split}
	\end{align} 
	and in view of Lemma
	\ref{lem:curvature_measure_euclidean_coordinates} and Fubini's theorem
	we may write 
	\begin{align*}
		\dcura_q(K,\,&\Sph^{n-1}\cap L)\\ 
		&=\frac{q}{n} \int\limits_{K|L} \left(\,\int\limits_{\conv\{ \vnull,K_{\ov{\vy}} \}\cap(\vy+L^\perp) } |\vz|^{q-n} \diff\mathcal{H}^{n-k}(\vz)\right)\diff\mathcal{H}^k(\vy)\\
		&=\frac{q}{n} \int\limits_{K|L} \left(\,\int\limits_{\rho_{K|L}(\vy)^{-1}K_{\ov{\vy}}} |\vz|^{q-n} \diff\mathcal{H}^{n-k}(\vz)\right)\diff\mathcal{H}^k(\vy)\\
		&=\frac{q}{n} \int\limits_{K|L} \rho_{K|L}(\vy)^{k-q}\left( \int\limits_{K_{\ov{\vy}}} |\vz|^{q-n} \diff\mathcal{H}^{n-k}(\vz)\right)\diff\mathcal{H}^k(\vy).
	\end{align*}
	
	The inner integral is independent of the length of $\vy\in
        K|L$ and might be as well considered as the value $g(\vu)$ of
        a (measurable) function $g\colon \Sph^{n-1}\cap L\to\R_{\geq
          0}$, i.e., 
        \begin{equation*}
           g(\vu)=\int\limits_{K_{\ov{\vu}}} |\vz|^{q-n} \diff\mathcal{H}^{n-k}(\vz).
        \end{equation*}
 With this notation and using spherical coordinates we obtain
	\begin{align}
	\begin{split}
		\dcura_q(K,\,&\Sph^{n-1}\cap L) \\ =& \frac{q}{n}\int\limits_{K|L} \rho_{K|L}(\vy)^{k-q} g(\vy/|\vy|) \diff\mathcal{H}^k(\vy)\\
		=&\frac{q}{n}\int\limits_{\Sph^{n-1}\cap L} g(\vu)
		\left(\int\limits_0^{\rho_{K|L}(\vu)}\rho_{K|L}(r\vu)^{k-q}
		r^{k-1} \mathrm{d}r\right)
		\diff\mathcal{H}^{k-1}(\vu)\\
		=&\frac{q}{n}\int\limits_{\Sph^{n-1}\cap L} g(\vu) \rho_{K|L}(\vu)^{k-q}
		\left(\int\limits_0^{\rho_{K|L}(\vu)}r^{q-1} \mathrm{d}r\right)
		\diff\mathcal{H}^{k-1}(\vu)\\
		=&\frac{1}{n}\int\limits_{\Sph^{n-1}\cap L} g(\vu)
		\,\rho_{K|L}(\vu)^k \diff\mathcal{H}^{k-1}(\vu).
	\end{split}
	\label{eq:numerator}
	\end{align}
	Applying the same transformation to the right hand side of
	\eqref{eq:lower_bound} leads to 

	\begin{equation}
	\begin{split}
		\dcura_q(K,\Sph^{n-1}) &>\\ & \frac{q}{n}\int\limits_{K|L} \rho_{K|L}(\vy)^{n-q} g(\vy/|\vy|) \diff\mathcal{H}^k(\vy)\\
		=& \frac{q}{n}\int\limits_{\Sph^{n-1}\cap L} g(\vu)
                \left(\,\int\limits_0^{\rho_{K|L}(\vu)}
                  \rho_{K|L}(r\vu)^{n-q} r^{k-1} \mathrm{d}r\right)
                \diff\mathcal{H}^{k-1}(\vu)\\
		=& \frac{q}{n}\int\limits_{\Sph^{n-1}\cap L} g(\vu) \rho_{K|L}(\vu)^{n-q} \left(\,\int\limits_0^{\rho_{K|L}(\vu)} r^{q-n+k-1} \mathrm{d}r\right) \diff\mathcal{H}^{k-1}(\vu)\\
		=& \frac{q}{n}\frac{1}{q-n+k}\int\limits_{\Sph^{n-1}\cap L} g(\vu)
		\,\rho_{K|L}(\vu)^k \diff\mathcal{H}^{k-1}(\vu).
	\end{split}
	\label{eq:denominator}
	\end{equation}
	Combining \eqref{eq:numerator} and \eqref{eq:denominator}
        gives the desired bound 
	\begin{equation}
	\frac{	\dcura_q(K,\Sph^{n-1}\cap
		L)}{\dcura_q(K,\Sph^{n-1})} < \frac{q-n+k}{q}.
	\end{equation}
\end{proof}

Next we show that the bounds given in Theorem~\ref{thm:subspace_bound}
are tight for every choice of $q\geq n+1$.

\begin{proof}[Proof of Proposition~\ref{prop:bound_opt}.]
Let $k\in\{1,\dots,n-1\}$ and for $l\in\N$ let $K_l$ be the cylinder
given as the cartesian product of two lower-dimensional ball 
\begin{equation}
K_l=(lB_k)\times B_{n-k}.
\end{equation} 
Let  $L=\lin \{ \ve_1,\ldots,\ve_k \}$ be the $k$-dimensional subspace
generated by  the first $k$ canonical unit vectors $\ve_i$. 
For $\vx\in \R^n$ write $\vx=\vx_1+\vx_2$, where
$\vx_1=\vx|L$ and $\vx_2=\vx|L^\perp$. The supporting hyperplane of $K_l$
with respect to  a unit vector $\vv\in \Sph^{n-1}\cap L$ is given by
\begin{equation}
	H_{K_l}(\vv)=\{ \vx\in \R^n\colon \langle \vv,\vx_1\rangle=l \}.
\end{equation}
Hence the part of the boundary of $K_l$ covered by all these supporting
hyperplanes is given by $l\Sph^{k-1}\times B_{n-k}$. 
In view of Lemma~\ref{lem:curvature_measure_euclidean_coordinates} and
Fubini's theorem we conclude 
\begin{equation}
\begin{split} 
	\dcura_q(&K_l, \Sph^{n-1}\cap L) =\\  &\frac{q}{n} \int\limits_{l\,B_k} \left(\,\int\limits_{{\{\vx_2\in B_{n-k}:\,l|\vx_2|\leq |\vx_1|\}}} (|\vx_1|^2+|\vx_2|^2)^{\frac{q-n}{2}} \diff\mathcal{H}^{n-k}(\vx_2) \right)\diff\mathcal{H}^{k}(\vx_1).
\end{split}
\label{eq:curvature_measure_cylinder_fubini}
\end{equation}
Denote the volume of $B_n$ by $\omega_n$. Recall, that the surface
area of $B_n$ is given by $n\omega_n$ and for abbreviation we set  
\begin{equation} 
c=c(q,k,n)=\frac{q}{n} k\omega_k (n-k)\omega_{n-k}.
\end{equation}
Switching to the cylindrical coordinates 
\begin{equation}
\vx_1=s\vu,\quad s\geq 0, \vu\in \Sph^{k-1},\qquad 
\vx_2=t\vv,\quad t\geq 0, \vv\in \Sph^{n-k-1},
\end{equation}
transforms the right hand side of
\eqref{eq:curvature_measure_cylinder_fubini} to 
\begin{align}
	\notag
	\dcura_q(K_l,\Sph^{n-1}\cap L) &= c \int\limits_0^l \int\limits_0^{s/l} s^{k-1} t^{n-k-1} (s^2+t^2)^{\frac{q-n}{2}}\diff t\diff s\\
	\notag
	&= c \int\limits_0^l \int\limits_0^1 s^{q-1} l^{k-n} t^{n-k-1} (1+l^{-2}t^2)^{\frac{q-n}{2}} \mathrm{d}t\diff s\\
	&= c\, l^{q-n+k} \int\limits_0^1 \int\limits_0^1 s^{q-1} t^{n-k-1} (1+l^{-2}t^2)^{\frac{q-n}{2}} \mathrm{d}t\diff s.
	\label{eq:curvature_measure_cylinder_L}
\end{align}
Analogously we obtain
\begin{equation}
\begin{split} 
	\dcura_q(K_l,\Sph^{n-1}) 
	& = \frac{q}{n} \int\limits_{\vx_1\in lB_k}
	\left(\,\int\limits_{\vx_2\in B_{n-k}}
	(|\vx_1|^2+|\vx_2|^2)^{\frac{q-n}{2}}
	\diff\mathcal{H}^{n-k}(\vx_2)
	\right)\diff\mathcal{H}^{k}(\vx_1) \\
	&  = c \int\limits_0^l \int\limits_0^{1} s^{k-1} t^{n-k-1} (s^2+t^2)^{\frac{q-n}{2}}\diff t\diff s\\
	& = c\, l^k \int\limits_0^1 \int\limits_0^1 s^{k-1} t^{n-k-1} (l^2s^2+t^2)^{\frac{q-n}{2}} \mathrm{d}t\diff s\\
	& = c\, l^{q-n+k} \int\limits_0^1 \int\limits_0^1 s^{k-1} t^{n-k-1} (s^2+l^{-2}t^2)^{\frac{q-n}{2}} \mathrm{d}t\diff s.
\end{split} 
\label{eq:curvature_measure_cylinder}
\end{equation}
The monotone convergence theorem gives
\begin{equation}
\begin{split} 
\lim_{l\to \infty} \dcura_q(K_l,\Sph^{n-1}\cap L) & =
\lim_{l\to \infty} \int\limits_0^1 \int\limits_0^1 s^{q-1} t^{n-k-1}
(1+l^{-2}t^2)^{\frac{q-n}{2}} \mathrm{d}t\diff s \\ &= \int\limits_0^1
s^{q-1} \mathrm{d}s \cdot \int\limits_0^1 t^{n-k-1} \mathrm{d}t =
\frac{1}{q(n-k)}
\end{split} 
\end{equation}
and
\begin{equation}
\begin{split}
\lim_{l\to \infty}\dcura_q(K_l,\Sph^{n-1}) & = 
	\lim_{l\to \infty} \int\limits_0^1 \int\limits_0^1 s^{k-1}
        t^{n-k-1} (s^2+l^{-2}t^2)^{\frac{q-n}{2}} \mathrm{d}t\diff s
        \\ 
&= \int\limits_0^1 s^{q-n+k-1} \mathrm{d}s \cdot \int\limits_0^1 t^{n-k-1} \mathrm{d}t= \frac{1}{(q-n+k)(n-k)}.
\end{split}
\end{equation}
Hence,  
\begin{equation}
\lim_{l\to \infty}\frac{\dcura_q(K_l,\Sph^{n-1}\cap L)}{\dcura_q(K_l,\Sph^{n-1})} = \frac{q-n+k}{q}.\qedhere
\end{equation}
\end{proof}

\section{The case $n<q<n+1$}
The only remaining open range of $q$  regarding the existence of a
subspace bound on the $q$th dual curvature 
of symmetric convex bodies is when $q\in(n,n+1)$. It is apparent from
the proof of 
Theorem~\ref{thm:subspace_bound} that an extension of
Theorem~\ref{thm:dual_curvature_measure_combination} 
to $p\in(0,1)$ would suffice. However, there are examples even in the $1$-dimensional case showing that this is not possible.
\begin{proposition}
	Let $p\in(0,1)$. There is an interval $K\subset\R$ and $\lambda\in[0,1]$ such that for $K_\lambda=\lambda K+(1-\lambda)(-K)$ we have
	\begin{equation}
		\int_{K_\lambda} |x|^p \diff x < |2\lambda-1|^p \int_K |x|^p \diff x.
	\end{equation}
\end{proposition}
\begin{proof}
	Let $\varepsilon>0$, $K=[\varepsilon,\varepsilon+1]$ and $\lambda=\frac{\varepsilon+1}{2\varepsilon+1}>\frac{1}{2}$ so that $K_\lambda=[0,1]$. Then
	\begin{equation}
		(p+1) \int_{K_\lambda} |x|^p \diff x = 1
	\end{equation}
	and
	\begin{align*}
		(p+1) (2\lambda-1)^p \int_K |x|^p \diff x &= (p+1)(2\varepsilon+1)^{-p} \int_{\varepsilon}^{\varepsilon+1} x^p \diff x\\
		&= \frac{(\varepsilon+1)^{p+1}-\varepsilon^{p+1}}{(2\varepsilon+1)^p}.
	\end{align*}
	Let $f(t)=(1+t)^{p+1}-t^{p+1}$ and $g(t)=(2t+1)^p$. Since
	\begin{gather*}
		f(0)=1=g(0),\\
		f'(0)=(p+1)>2p = g'(0),
	\end{gather*}
	we have
	\begin{equation}
		\frac{(\varepsilon+1)^{p+1}-\varepsilon^{p+1}}{(2\varepsilon+1)^p}>1
	\end{equation}
	for small $\varepsilon$ and hence, 
\begin{equation}
		\int_{K_\lambda} |x|^p \diff x < (2\lambda-1)^p \int_K |x|^p \diff x.
	\end{equation}
\end{proof}
Nevertheless, the examples given in the proof of
Proposition~\ref{prop:bound_opt} indicate that the subspace bound
\eqref{eq:subspace_bound} might also be correct for $q\in(n,n+1)$. Next
we will verify this in the special case of  parallelotopes and to this
end we  need the next lemma about integrals of quasiconvex
functions $f:\R^n\to\R$, i.e., functions whose sublevel sets 
\begin{equation} 
L_f^-(\alpha)=\{\vx\in\R^n : f(\vx) \leq \alpha \}, \alpha\in\R, 
\end{equation} 
 are convex. The lemma follows directly from 
 \cite[Theorem~1]{Anderson:1955} where the reverse inequality is
proved for quasiconcave functions and superlevel sets.  
\begin{lemma}
	\label{lem:quasiconvex_integral}
	Let $f\colon\R^n\to\R_{\geq 0}$ be an even, quasiconvex function, that is integrable on convex, compact sets. Let $K\in\Kn$ with $K=-K$ and $\dim K=k$. Let $\lambda\in [0,1]$ and $\vv\in\R^n$. Then 
	\begin{equation}
		\label{eq:integral_conv_comb}
		\int_{K+\lambda\vv} f(\vx)\diff\mathcal{H}^k(\vx) \leq \int_{K+\vv} f(\vx)\diff\mathcal{H}^k(\vx).
	\end{equation}
	Moreover, equality holds for $\lambda\in [0,1)$ if and only if for every $\alpha>0$
	\begin{equation}
		(K+\vv)\cap L_f^-(\alpha) = (K\cap L_f^-(\alpha))+\vv.
	\end{equation}
\end{lemma}
\begin{proof}  Apply \cite[Theorem 1]{Anderson:1955} to the function $\tilde{f}(\vx)=\max(c-f(\vx),0)$ where $c=\sup f(K+\vv)$.
\end{proof}

Based on Lemma \ref{lem:quasiconvex_integral} we can easily get a
lower bound on the subspace concentration of prisms.
\begin{proposition}
	\label{prop:subspace_concentration_parallelotopes_lb}
	Let $\vu\in \Sph^{n-1}$ and let $Q\in\Kn$ with $Q=-Q$, $Q\subset \vu^\perp$ and $\dim Q=n-1$. 
        Let $\vv\in\R^n\setminus \aff Q$ and let $P$ be the prism  $P=\conv(Q-\vv,Q+\vv)$. Then
        for $q>n$ 
	\begin{equation}
		\label{eq:subspace_concentration_lb}
		\tilde{C}_q(P,\{\vu, -\vu \}) > \frac{1}{q} \tilde{C}_q(P,\Sph^{n-1}).
	\end{equation}
\end{proposition}
\begin{proof} Let  $L=\lin\{ \vu \}$. 
	Since the dual curvature measure is homogeneous we may assume
        that $\langle\vu,\vv\rangle=1$. By Lemma~\ref{lem:curvature_measure_euclidean_coordinates} we may write
	\begin{align*}
		\tilde{C}_q(P,\Sph^{n-1}) &= \frac{q}{n} \int_{P|L} \int_{P\cap (\vy+L^\perp)} |\vz|^{q-n}\diff\mathcal{H}^{n-1}(\vz)\diff\mathcal{H}(\vy)\\
		&= \frac{q}{n} \int_{-1}^1 \int_{Q+\tau\vv} |\vz|^{q-n}\diff\mathcal{H}^{n-1}(\vz)\diff \tau\\
		&= 2\frac{q}{n} \int_0^1 \int_{Q+\tau\vv} |\vz|^{q-n}\diff\mathcal{H}^{n-1}(\vz)\diff \tau.
	\end{align*}
	Applying Lemma~\ref{lem:quasiconvex_integral} to the inner
        integral gives 
	\begin{align*}
	\tilde{C}_q(P,\Sph^{n-1}) &\leq 2\frac{q}{n}\int_0^1 \int_{Q+\vv} |\vz|^{q-n}\diff\mathcal{H}^{n-1}(\vz)\diff \tau\\
	&= 2 \frac{q}{n} \int_{Q+v} |\vz|^{q-n}\diff\mathcal{H}^{n-1}(\vz).
	\end{align*}
	On the other hand, 
	\begin{align*}
	\tilde{C}_q(P,\{\vu, -\vu\}) &= 2\frac{q}{n} \int_0^1 \int_{\tau(Q+\vv)} |\vz|^{q-n}\diff\mathcal{H}^{n-1}(\vz)\diff \tau\\
	&= 2\frac{q}{n} \int_0^1 \int_{Q+\vv} \tau^{q-n} |\vz|^{q-n} \cdot\tau^{n-1}\diff\mathcal{H}^{n-1}(\vz) \diff \tau\\
	&= 2\frac{q}{n} \int_{Q+\vv} |\vz|^{q-n} \diff\mathcal{H}^{n-1}(\vz) \int_0^1 \tau^{q-1} \diff \tau\\
	&= 2 \frac{1}{n} \int_{Q+\vv} |\vz|^{q-n} \diff\mathcal{H}^{n-1}(\vz).
	\end{align*}
	This gives \eqref{eq:subspace_concentration_lb} without strict
        inequality. Suppose we have equality. Then  the equality
        characterization of Lemma~\ref{lem:quasiconvex_integral}
        implies  
	\[ (Q+\vv)\cap rB_n = (Q\cap r B_n)+\vv \]
	for almost all $r>0$. But for small $r$ the left hand side is empty and hence equality in \eqref{eq:subspace_concentration_lb} cannot be attained.
\end{proof}

As a consequence we deduce an upper bound on the subspace concentration of dual curvature measures of parallelotopes.

\begin{corollary}
	\label{cor:subspace_concentration_parallelotopes_ub}
	Let $P\in\Ken$ be a parallelotope, and let $L\subset \R^n$ be
        a proper subspace of $\R^n$. Then for $q>n$ 
	\begin{equation}
	\frac{\tilde{C}_q(P,\Sph^{n-1}\cap L)}{\tilde{C}_q(P,\Sph^{n-1})} < \frac{q-n+\dim L}{q}.
	\end{equation}
\end{corollary}
\begin{proof}
	Let $\pm \vu_1,\ldots,\pm \vu_n\in \Sph^{n-1}$ be the outer normal vectors of $P$. In particular, $\vu_1,\ldots,\vu_n$ are linearly independent and by Proposition~\ref{prop:subspace_concentration_parallelotopes_lb}
	\begin{align*}
	\tilde{C}_q(P,\Sph^{n-1}\cap L) &= \sum_{\vu_i\in L}
                                          \tilde{C}_q(P,\{\vu_i, -\vu_i \})\\
	&= \tilde{C}_q(P,\Sph^{n-1})-\sum_{\vu_i\notin L}
          \tilde{C}_q(P,\{\vu_i, -\vu_i \})\\
	&< \tilde{C}_q(P,\Sph^{n-1})-\sum_{\vu_i\notin L} \frac{1}{q}\tilde{C}_q(P,\Sph^{n-1}).
	\end{align*}
	$L$ can contain at most $\dim L$ of the $\vu_i$'s and therefore
	\begin{align*}
	\tilde{C}_q(P,\Sph^{n-1}\cap L) &< \tilde{C}_q(P,\Sph^{n-1})-\frac{n-\dim L}{q}\tilde{C}_q(P,\Sph^{n-1})\\
	&= \frac{q-n+\dim L}{q} \tilde{C}_q(P,\Sph^{n-1}).
	\end{align*}
\end{proof}

In particular this settles the $2$-dimensional case as it can be reduced to proving a subspace bound for parallelograms.

\begin{proof}[Proof of Theorem~\ref{thm:subspace_bound_plane}]
	Let $L=\lin\{\vu\}$, $\vu\in \Sph^1$. If $F=K\cap
        H(\vu,h_K(\vu))$ is a singleton, inequality \eqref{eq:subspace_concentration_plane} trivially holds. Assume $\dim F=1$. By an inclusion argument it suffices to prove \eqref{eq:subspace_concentration_plane} for $P=\conv(F\cup (-F))$. Since $F$ is a line segment, $P$ is a parallelogram and Corollary~\ref{cor:subspace_concentration_parallelotopes_ub} gives \eqref{eq:subspace_concentration_plane} for $P$.
\end{proof}

\section*{Acknowledgements}
The authors thank K{\'a}roly~J. B{\"o}r{\"o}czky, Apostolos Giannopoulos and Yiming Zhao for their very helpful comments and suggestions.


\begin{thebibliography}{99}
	
	\bibitem{Alesker:1995}
	Semyon Alesker.
	\newblock {$\psi_2$}-estimate for the {E}uclidean norm on a convex body in
	isotropic position.
	\newblock In {\em Geometric aspects of functional analysis ({I}srael,
		1992--1994)}, volume~77 of {\em Oper. Theory Adv. Appl.}, pages 1--4.
	Birkh\"auser, Basel, 1995.
	
	\bibitem{Anderson:1955}
	Theodore~Wilbur Anderson.
	\newblock The integral of a symmetric unimodal function over a symmetric convex
	set and some probability inequalities.
	\newblock {\em Proc. Amer. Math. Soc.}, 6:170--176, 1955.
	
	\bibitem{ArtsteinGiannopoulosMilman:2015}
	Shiri Artstein-Avidan, Apostolos Giannopoulos, and Vitali~D. Milman.
	\newblock {\em Asymptotic geometric analysis. {P}art {I}}, volume 202 of {\em
		Mathematical Surveys and Monographs}.
	\newblock American Mathematical Society, Providence, RI, 2015.
	
	\bibitem{Barthe:2005}
	Franck Barthe, Olivier Gu{\'e}don, Shahar Mendelson, and Assaf Naor.
	\newblock A probabilistic approach to the geometry of the {$l^n_p$}-ball.
	\newblock {\em Ann. Probab.}, 33(2):480--513, 2005.
	
	\bibitem{Boroczky:2015c}
	K{\'a}roly~J. B{\"o}r{\"o}czky and P{\'a}l Heged{\H{u}}s.
	\newblock The cone volume measure of antipodal points.
	\newblock {\em Acta Math. Hungar.}, 146(2):449--465, 2015.
	
	\bibitem{Boroczky:2015b}
	K{\'a}roly~J. B{\"o}r{\"o}czky, P{\'a}l Heged{\H{u}}s, and Guangxian Zhu.
	\newblock On the discrete logarithmic minkowski problem.
	\newblock {\em International Mathematics Research Notices}, 2015.
	
	\bibitem{Boroczky:2016}
	K{\'a}roly~J. B{\"o}r{\"o}czky and Martin Henk.
	\newblock Cone-volume measure of general centered convex bodies.
	\newblock {\em Adv. Math.}, 286:703--721, 2016.
	
	\bibitem{Boroczky:2017a}
	K{\'a}roly~J. B{\"o}r{\"o}czky and Martin Henk.
	\newblock Cone-volume measure and stability.
	\newblock {\em Adv. Math.}, in press.
	
	\bibitem{BoroczkyHenkPollehn:2017}
	K{\'a}roly~J. B{\"o}r{\"o}czky, Martin Henk, and Hannes Pollehn.
	\newblock Subspace concentration of dual curvature measures.
	\newblock {\em J. Differential Geom.}, in press.
	
	\bibitem{Boroczky:2012}
	K{\'a}roly~J. B{\"o}r{\"o}czky, Erwin Lutwak, Deane Yang, and Gaoyong Zhang.
	\newblock The log-{B}runn-{M}inkowski inequality.
	\newblock {\em Adv. Math.}, 231(3-4):1974--1997, 2012.
	
	\bibitem{Boroczky:2013}
	K{\'a}roly~J. B{\"o}r{\"o}czky, Erwin Lutwak, Deane Yang, and Gaoyong Zhang.
	\newblock The logarithmic {M}inkowski problem.
	\newblock {\em J. Amer. Math. Soc.}, 26(3):831--852, 2013.
	
	\bibitem{Boroczky:2015a}
	K{\'a}roly~J. B{\"o}r{\"o}czky, Erwin Lutwak, Deane Yang, and Gaoyong Zhang.
	\newblock Affine images of isotropic measures.
	\newblock {\em J. Differential Geom.}, 99(3):407--442, 2015.
	
	\bibitem{Boroczky:2017b}
	K{\'a}roly~J. B{\"o}r{\"o}czky, Erwin Lutwak, Deane Yang, Gaoyong Zhang, and
	Yiming Zhao.
	\newblock The dual {M}inkowski problem for symmetric convex bodies.
	\newblock preprint.
	
	\bibitem{Caffarelli:1990}
	Luis~A. Caffarelli.
	\newblock Interior {$W^{2,p}$} estimates for solutions of the
	{M}onge-{A}mp\`ere equation.
	\newblock {\em Ann. of Math. (2)}, 131(1):135--150, 1990.
	
	\bibitem{Chen:2017}
	Shibing Chen, Qi-Rui Li, and Guangxian Zhu.
	\newblock The logarithmic {M}inkowski problem for non-symmetric measures.
	\newblock preprint.
	
	\bibitem{Cheng:1976}
	Shiu~Yuen Cheng and Shing~Tung Yau.
	\newblock On the regularity of the solution of the {$n$}-dimensional
	{M}inkowski problem.
	\newblock {\em Comm. Pure Appl. Math.}, 29(5):495--516, 1976.
	
	\bibitem{Chou:2006}
	Kai-Seng Chou and Xu-Jia Wang.
	\newblock The {$L_p$}-{M}inkowski problem and the {M}inkowski problem in
	centroaffine geometry.
	\newblock {\em Adv. Math.}, 205(1):33--83, 2006.
	
	\bibitem{Firey:1962}
	William~J. Firey.
	\newblock {$p$}-means of convex bodies.
	\newblock {\em Math. Scand.}, 10:17--24, 1962.
	
	\bibitem{Gardner:1994}
	Richard~J. Gardner.
	\newblock A positive answer to the {B}usemann-{P}etty problem in three
	dimensions.
	\newblock {\em Ann. of Math. (2)}, 140(2):435--447, 1994.
	
	\bibitem{Gardner:2002}
	Richard~J. Gardner.
	\newblock The {B}runn-{M}inkowski inequality.
	\newblock {\em Bull. Amer. Math. Soc. (N.S.)}, 39(3):355--405, 2002.
	
	\bibitem{Gardner:2006}
	Richard~J. Gardner.
	\newblock {\em Geometric tomography}, volume~58 of {\em Encyclopedia of
		Mathematics and its Applications}.
	\newblock Cambridge University Press, Cambridge, second edition, 2006.
	
	\bibitem{Gardner:2013}
	Richard~J. Gardner, Daniel Hug, and Wolfgang Weil.
	\newblock Operations between sets in geometry.
	\newblock {\em J. Eur. Math. Soc. (JEMS)}, 15(6):2297--2352, 2013.
	
	\bibitem{Gardner:2014}
	Richard~J. Gardner, Daniel Hug, and Wolfgang Weil.
	\newblock The {O}rlicz-{B}runn-{M}inkowski theory: a general framework,
	additions, and inequalities.
	\newblock {\em J. Differential Geom.}, 97(3):427--476, 2014.
	
	\bibitem{GardnerKoldobskySchlumprecht:1999}
	Richard~J. Gardner, Alexander Koldobsky, and Thomas Schlumprecht.
	\newblock An analytic solution to the {B}usemann-{P}etty problem on sections of
	convex bodies.
	\newblock {\em Ann. of Math. (2)}, 149(2):691--703, 1999.
	
	\bibitem{Gromov:1987}
	Michail Gromov and Vitali~D. Milman.
	\newblock Generalization of the spherical isoperimetric inequality to uniformly
	convex {B}anach spaces.
	\newblock {\em Compositio Math.}, 62(3):263--282, 1987.
	
	\bibitem{Gruber:2007}
	Peter~M. Gruber.
	\newblock {\em Convex and discrete geometry}, volume 336 of {\em Grundlehren
		der Mathematischen Wissenschaften [Fundamental Principles of Mathematical
		Sciences]}.
	\newblock Springer, Berlin, 2007.
	
	\bibitem{Haberl:2014}
	Christoph Haberl and Lukas Parapatits.
	\newblock The centro-affine {H}adwiger theorem.
	\newblock {\em J. Amer. Math. Soc.}, 27(3):685--705, 2014.
	
	\bibitem{He:2006}
	Binwu He, Gangsong Leng, and Kanghai Li.
	\newblock Projection problems for symmetric polytopes.
	\newblock {\em Adv. Math.}, 207(1):73--90, 2006.
	
	\bibitem{Henk:2014}
	Martin Henk and Eva Linke.
	\newblock Cone-volume measures of polytopes.
	\newblock {\em Adv. Math.}, 253:50--62, 2014.
	
	\bibitem{Henk:2005}
	Martin Henk, Achill Sch{\"u}rmann, and J{\"o}rg~M. Wills.
	\newblock Ehrhart polynomials and successive minima.
	\newblock {\em Mathematika}, 52(1-2):1--16 (2006), 2005.
	
	\bibitem{Huang:2016}
	Yong Huang, Erwin Lutwak, Deane Yang, and Gaoyong Zhang.
	\newblock Geometric measures in the dual {B}runn-{M}inkowski theory and their
	associated {M}inkowski problems.
	\newblock {\em Acta Math.}, 216(2):325--388, 2016.
	
	\bibitem{Hug:2005}
	Daniel Hug, Erwin Lutwak, Deane Yang, and Gaoyong Zhang.
	\newblock On the {$L_p$} {M}inkowski problem for polytopes.
	\newblock {\em Discrete Comput. Geom.}, 33(4):699--715, 2005.
	
	\bibitem{Kadelburg:2005}
	Zoran Kadelburg, Du{\v s}an {\DJ}uki{\'c}, Milivoje Luki{\'c}, and Ivan
	Mati{\'c}.
	\newblock Inequalities of {Karamata}, {Schur} and {Muirhead}, and some
	applications.
	\newblock {\em The Teaching of Mathematics}, 8(1):31--45, 2005.
	
	\bibitem{Ludwig:2010a}
	Monika Ludwig.
	\newblock General affine surface areas.
	\newblock {\em Adv. Math.}, 224(6):2346--2360, 2010.
	
	\bibitem{Ludwig:2010b}
	Monika Ludwig and Matthias Reitzner.
	\newblock A classification of {${\rm SL}(n)$} invariant valuations.
	\newblock {\em Ann. of Math. (2)}, 172(2):1219--1267, 2010.
	
	\bibitem{Lutwak:1975b}
	Erwin Lutwak.
	\newblock Dual cross-sectional measures.
	\newblock {\em Atti Accad. Naz. Lincei Rend. Cl. Sci. Fis. Mat. Natur. (8)},
	58(1):1--5, 1975.
	
	\bibitem{Lutwak:1975a}
	Erwin Lutwak.
	\newblock Dual mixed volumes.
	\newblock {\em Pacific J. Math.}, 58(2):531--538, 1975.
	
	\bibitem{Lutwak:1993}
	Erwin Lutwak.
	\newblock The {B}runn-{M}inkowski-{F}irey theory. {I}. {M}ixed volumes and the
	{M}inkowski problem.
	\newblock {\em J. Differential Geom.}, 38(1):131--150, 1993.
	
	\bibitem{Lutwak:2005}
	Erwin Lutwak, Deane Yang, and Gaoyong Zhang.
	\newblock {$L_p$} {J}ohn ellipsoids.
	\newblock {\em Proc. London Math. Soc. (3)}, 90(2):497--520, 2005.
	
	\bibitem{Lutwak:2010a}
	Erwin Lutwak, Deane Yang, and Gaoyong Zhang.
	\newblock Orlicz centroid bodies.
	\newblock {\em J. Differential Geom.}, 84(2):365--387, 2010.
	
	\bibitem{Lutwak:2010b}
	Erwin Lutwak, Deane Yang, and Gaoyong Zhang.
	\newblock Orlicz projection bodies.
	\newblock {\em Adv. Math.}, 223(1):220--242, 2010.
	
	\bibitem{Ma:2015}
	Lei Ma.
	\newblock A new proof of the log-{B}runn-{M}inkowski inequality.
	\newblock {\em Geom. Dedicata}, 177:75--82, 2015.
	
	\bibitem{Naor:2007}
	Assaf Naor.
	\newblock The surface measure and cone measure on the sphere of {$l_p^n$}.
	\newblock {\em Trans. Amer. Math. Soc.}, 359(3):1045--1079 (electronic), 2007.
	
	\bibitem{Naor:2003}
	Assaf Naor and Dan Romik.
	\newblock Projecting the surface measure of the sphere of {$l_p^n$}.
	\newblock {\em Ann. Inst. H. Poincar\'e Probab. Statist.}, 39(2):241--261,
	2003.
	
	\bibitem{Nirenberg:1953}
	Louis Nirenberg.
	\newblock The {W}eyl and {M}inkowski problems in differential geometry in the
	large.
	\newblock {\em Comm. Pure Appl. Math.}, 6:337--394, 1953.
	
	\bibitem{Paouris:2003}
	Grigoris Paouris.
	\newblock {$\Psi_2$}-estimates for linear functionals on zonoids.
	\newblock In {\em Geometric aspects of functional analysis}, volume 1807 of
	{\em Lecture Notes in Math.}, pages 211--222. Springer, Berlin, 2003.
	
	\bibitem{Paouris:2012}
	Grigoris Paouris and Elisabeth~M. Werner.
	\newblock Relative entropy of cone measures and {$L_p$} centroid bodies.
	\newblock {\em Proc. Lond. Math. Soc. (3)}, 104(2):253--286, 2012.
	
	\bibitem{Schneider:1993}
	Rolf Schneider.
	\newblock {\em Convex bodies: the {B}runn-{M}inkowski theory}, volume~44 of
	{\em Encyclopedia of Mathematics and its Applications}.
	\newblock Cambridge University Press, Cambridge, 1993.
	
	\bibitem{Stancu:2002}
	Alina Stancu.
	\newblock The discrete planar {$L_0$}-{M}inkowski problem.
	\newblock {\em Adv. Math.}, 167(1):160--174, 2002.
	
	\bibitem{Stancu:2003}
	Alina Stancu.
	\newblock On the number of solutions to the discrete two-dimensional
	{$L_0$}-{M}inkowski problem.
	\newblock {\em Adv. Math.}, 180(1):290--323, 2003.
	
	\bibitem{Stancu:2012}
	Alina Stancu.
	\newblock Centro-affine invariants for smooth convex bodies.
	\newblock {\em Int. Math. Res. Not. IMRN}, (10):2289--2320, 2012.
	
	\bibitem{Zhang:1999}
	Gaoyong Zhang.
	\newblock A positive solution to the {B}usemann-{P}etty problem in {$\mathbf
		R^4$}.
	\newblock {\em Ann. of Math. (2)}, 149(2):535--543, 1999.
	
	\bibitem{Zhao:2017a}
	Yiming Zhao.
	\newblock The dual {M}inkowski problem for negative indices.
	\newblock {\em Calc. Var. Partial Differential Equations}, 56(2):56:18, 2017.
	
	\bibitem{Zhao:2017b}
	Yiming Zhao.
	\newblock Existence of solutions to the even dual {M}inkowski problem.
	\newblock {\em J. Differential Geom.}, in press.
	
	\bibitem{Zhu:2014}
	Guangxian Zhu.
	\newblock The logarithmic {M}inkowski problem for polytopes.
	\newblock {\em Adv. Math.}, 262:909--931, 2014.
	
	\bibitem{Zhu:2015}
	Guangxian Zhu.
	\newblock The centro-affine {M}inkowski problem for polytopes.
	\newblock {\em J. Differential Geom.}, 101(1):159--174, 2015.
	
	\bibitem{Zhu:2015a}
	Guangxian Zhu.
	\newblock The {$L_p$} {M}inkowski problem for polytopes for {$0<p<1$}.
	\newblock {\em J. Funct. Anal.}, 269(4):1070--1094, 2015.
	
\end{thebibliography}

\end{document}